\documentclass[final]{siamltex}
\usepackage{amsmath,amsfonts,amssymb}
\usepackage{latexsym}
\usepackage{graphicx,overpic}
\usepackage{subfig}
\usepackage{pgfplots}
\pgfplotsset{compat=newest}
\usepackage{color}
\usepackage{booktabs}
\usepackage{float}

\newcommand{\R}{\mathbb R}
\newcommand{\bA}{\mathbf A}

\newcommand{\bH}{\mathbf H}
\newcommand{\bI}{\mathbf I}

\newcommand{\bP}{\mathbf P}

\newcommand{\bV}{\mathbf V}

\newcommand{\ba}{\mathbf a}
\newcommand{\bb}{\mathbf b}

\newcommand{\blf}{\mathbf f}
\newcommand{\bn}{\mathbf n}
\newcommand{\be}{\mathbf e}
\newcommand{\bp}{\mathbf p}
\newcommand{\br}{\mathbf r}

\newcommand{\bu}{\mathbf u}
\newcommand{\bU}{\mathbf U}
\newcommand{\bv}{\mathbf v}
\newcommand{\bw}{\mathbf w}

\newcommand{\bx}{\mathbf x}

\newcommand{\A}{\mathcal A}

\newcommand{\T}{\mathcal T}

\newcommand{\Div}{\mathop{\rm div}}
\newcommand{\divG}{{\mathop{\,\rm div}}_{\Gamma}}
\newcommand{\gradG}{\nabla_{\Gamma}}
\newcommand{\nablaG}{\nabla_{\Gamma}}

\newcommand{\OGamma}{\Omega^\Gamma_h}

\newcommand{\tr}{{\rm tr}}
\DeclareGraphicsExtensions{.pdf,.eps,.ps,.eps.gz,.ps.gz,.eps.Y}

\newcommand{\la}{\left\langle}
\newcommand{\ra}{\right\rangle}

\def\cl {\nonumber \\}
\def\el {\nonumber }

\newtheorem{assumption}{Assumption}[section]

\newtheorem{remark}{Remark}[section]

\begin{document}
\title{A finite element method for the surface Stokes problem}
\author{
Maxim A. Olshanskii\thanks{Department of Mathematics, University of Houston, Houston, Texas 77204 (molshan@math.uh.edu); Partially supported by NSF through the Division of Mathematical Sciences grants 1522252 and 1717516.}
\and Annalisa Quaini\thanks{Department of Mathematics, University of Houston, Houston, Texas 77204 (quaini@math.uh.edu); Partially supported by NSF through grant DMS-1263572 and DMS-1620384.},
\and Arnold Reusken\thanks{Institut f\"ur Geometrie und Praktische  Mathematik, RWTH-Aachen
University, D-52056 Aachen, Germany (reusken@igpm.rwth-aachen.de).}
\and
Vladimir Yushutin \thanks{Department of Mathematics, University of Houston, Houston, Texas 77204 (yushutin@math.uh.edu)}
}
\maketitle
%
\begin{abstract}
We consider a Stokes problem posed on a 2D surface embedded in a 3D domain.
The equations describe an equilibrium, area-preserving tangential flow of a viscous  surface fluid and
serve as a model problem in the dynamics of material interfaces.
In this paper, we develop and analyze a Trace finite element method (TraceFEM)
for such a surface Stokes problem. TraceFEM relies on finite element spaces defined on a fixed, surface-independent background mesh which consists of shape-regular tetrahedra.
Thus, there is no need for surface parametrization or surface fitting with the mesh.
The TraceFEM treated here is based on $P_1$ bulk finite elements for both the velocity and the pressure.
In order to enforce the velocity vector field to be tangential to the surface we introduce a penalty term.
The method  is straightforward to implement and  has an $O(h^2)$ geometric consistency error, which is of the same order as the approximation error due to  the $P_1$--$P_1$ pair for velocity and pressure.
We prove stability and optimal order discretization error bounds in the surface $H^1$ and $L^2$ norms.
A series of numerical experiments is presented to  illustrate certain features of the proposed TraceFEM.
\end{abstract}

\section{Introduction}
Fluid equations  on manifolds appear in the literature on modeling of emulsions, foams and  biological  membranes.
See, for example, \cite{scriven1960dynamics,slattery2007interfacial,arroyo2009,brenner2013interfacial, rangamani2013interaction,rahimi2013curved}.
They are also studied as an interesting mathematical problem in its own right in, e.g.,  \cite{ebin1970groups,Temam88,taylor1992analysis,arnol2013mathematical,mitrea2001navier, arnaudon2012lagrangian,Gigaetal}. Despite the apparent practical and  mathematical relevance,
so far fluids on manifolds have received little attention  from the scientific computing community.
Only few papers, like for instance \cite{holst2012geometric,nitschke2012finite,barrett2014stable,reuther2015interplay, ReuskenZhang,reuther2017solving,fries2017higher}, treat
the development and analysis of numerical  methods for surface fluid equations or coupled bulk--surface fluid problems.
Among those papers, \cite{reuther2017solving,fries2017higher} applied surface finite element methods to discretize
the incompressible surface Navier-Stokes equations in primitive variables on stationary manifolds.
In \cite{reuther2017solving}, the authors considered $P_1$-$P_1$ finite elements with no pressure stabilization
and a penalty technique to force the flow field to be tangential to the surface.
In \cite{fries2017higher}, instead, surface Taylor--Hood elements are used and combined with  a Lagrange multiplier method to enforce the tangentiality constraint.
Neither references address the numerical analysis of
finite element methods for the surface Navier-Stokes equations.
In general,
we are not aware of any paper containing a rigorous analysis of
finite element (or any other) discretization methods for surface (Navier-)Stokes equations.

In recent years,
several papers on discretization methods for \emph{scalar} elliptic and parabolic partial differential equations on surfaces
have appeared. We refer to~\cite{DEreview,olshanskii2016trace} for a review on surface finite element methods.
Only very recently finite element methods have been applied and analyzed for \emph{vector} Laplace equations on surfaces~\cite{hansbo2016analysis,gross2017trace}. This is a first natural step in extending the methods and analyses
proposed for the scalar problems to  surface (Navier--)Stokes equations.
In \cite{hansbo2016analysis},
the authors analyzed a surface FEM combined with a penalty technique to impose the tangentiality constraint.
The results include stability and error analysis, which also account for the effects of geometric errors.
The approach presented in \cite{gross2017trace} is different: an unfitted finite element method (TraceFEM)
combined with a Lagrange multiplier technique to enforce the discrete vector fields to be (approximately) tangential
to the surface. Stability and optimal order error estimates were also proved in \cite{gross2017trace}. The present
paper continues along this latter line of research and studies the TraceFEM applied to the
Stokes equations posed on a stationary closed smooth surface $\Gamma$ embedded in $\R^3$.

The choice of the geometrically  \emph{un}fitted discretization (instead of the surface FEM  as in
\cite{reuther2017solving,fries2017higher}) is motivated by our ultimate goal:
the numerical simulation of fluid flows on \emph{evolving} surfaces $\Gamma(t)$~\cite{arroyo2009,Gigaetal,Jankuhn1}, including  cases where a parametrization of $\Gamma(t)$ is not explicitly available and $\Gamma(t)$ may undergo large deformations or even topological changes. Unfitted discretizations, such as TraceFEM,
allow to avoid mesh reconstruction for the time-dependent geometry and
to treat implicitly defined surfaces. As illustrated in \cite{olshanskii2014eulerian,lehrenfeld2017stabilized},
TraceFEM works very well for scalar PDEs posed on evolving surfaces
and can be naturally combined with the level set method for (implicit) surface representation.
In~\cite{Jankuhn1}, it is shown that the surface (Navier--)Stokes equations used to model
incompressible surface fluid systems on  evolving surfaces
admit a natural splitting into coupled equations for  tangential and normal motions.
Such splitting and time discretization yield a subproblem that is very similar to the
Stokes problem treated in this paper.  Hence, we consider the detailed study
of a trace FEM for the  Stokes problem on a stationary surface to be an important step
in the development of a robust and efficient finite element solver
for the Navier-Stokes equations on evolving surfaces.
In addition, the method treated in this paper can be used to study interesting properties of
Stokes problems on stationary surfaces, as illustrated by the numerical experiments presented in section~\ref{sectExp}.

The TraceFEM considered in this paper is based on the $P_1$--$P_1$ finite element pair
defined on the background mesh. Pressure stabilization is achieved through the
simple Brezzi--Pitk\"{a}ranta stabilization.
Unlike \cite{gross2017trace}, we consider a penalty technique for the tangentiality condition.
Altogether, this results in a straightforward to use solver for fluid equations posed on surfaces.
An alternative approach with higher order (generalized) Taylor--Hood elements and
a Lagrange multiplier method will be subject of future work.

The principal contributions of  the paper are the following:
\begin{itemize}
\item[-] Introduction of an  easy to implement TraceFEM, which is based on
a piecewise planar approximation of the surface, $P_1$--$P_1$ finite elements on the background mesh, Brezzi--Pitk\"{a}ranta and so-called ``volume normal derivative''  stabilization terms, and a penalty method for the tangentiality constraint.
\item[-] Error analysis showing optimal order $O(h)$ error estimates in the $H^1(\Gamma)^3$ norm for velocity and $L^2(\Gamma)$ norm for pressure, and an optimal $O(h^2)$ error estimate
for velocity in the $L^2(\Gamma)^3$ norm.
All these estimates do not depend on the position of $\Gamma$ in the background mesh. The analysis does {not} account for the effect of geometric errors.
\item[-] Study of the conditioning of the resulting saddle point stiffness matrix.
We prove that the spectral condition number of this matrix is bounded by $c h^{-2}$, with a constant  $c$ that is independent of the position of the surface $\Gamma$ relative to the underlying
triangulation.
\item[-] Presentation of an optimal preconditioner.
\end{itemize}


The outline of the paper is as follows. In section~\ref{s_cont},
we introduce the surface Stokes system and some notions of  tangential differential calculus.
We give a weak formulation of the problem and recall some known results.
In section~\ref{s_pen}, we study the augmented surface Stokes  problem, i.e.
the problem with additional penalty terms.
An estimate on the difference between weak solutions of the original and augmented problem is derived.
An unfitted finite element method (TraceFEM) for the surface Stokes problem is introduced in section~\ref{s_TraceFEM}.
The role of the different stabilization terms is explained.
In section~\ref{s_error}, an error analysis of this method is presented.
A discrete inf-sup stability and optimal a-priori discretization error bounds are proved.
In section~\ref{s_cond}, we analyze the conditioning properties of the resulting saddle point matrix and an optimal preconditioner is introduced. Numerical results in section~\ref{sectExp} illustrate  the performance of the method in terms of error convergence,  efficiency of the linear solver, and flexibility in handling implicitly defined geometries.

\section{Continuous problem}\label{s_cont}
Assume that $\Gamma$ is a closed sufficiently smooth surface in $\mathbb{R}^3$. The outward pointing unit normal on $\Gamma$ is denoted by $\bn$, and the orthogonal projection on the tangential plane is given by $\bP=\bP(\bx):= \bI - \bn(\bx)\bn(\bx)^T$, $\bx \in \Gamma$. In a neighborhood $\mathcal{O}(\Gamma)$  of $\Gamma$ the closest point projection $\bp:\,\mathcal{O}(\Gamma)\to \Gamma$ is well defined.  For a scalar function $p:\, \Gamma \to \mathbb{R}$ or a vector function $\bu:\, \Gamma \to \mathbb{R}^3$  we define $p^e=p\circ \bp\,:\,\mathcal{O}(\Gamma)\to\mathbb{R}$, $\bu^e=\bu\circ \bp\,:\,\mathcal{O}(\Gamma)\to\mathbb{R}^3$,   extensions of $p$ and $\bu$ from $\Gamma$ to its neighborhood $\mathcal{O}(\Gamma)$ along the normal directions.
On $\Gamma$ it holds $\nabla p^e= \bP\nabla p^e$  and  $\nabla\bu^e=\nabla\bu^e\bP$, with $\nabla \bu:= (\nabla u_1~ \nabla u_2 ~\nabla u_3)^T \in \mathbb{R}^{3 \times 3}$ for vector functions $\bu$. The surface gradient and covariant derivatives on $\Gamma$ are then defined as $\nablaG p=\bP\nabla p^e$ and  $\nabla_\Gamma \bu:= \bP \nabla \bu^e \bP$. Note that the definitions  of surface gradient and covariant derivatives are  independent of a particular smooth extension of $p$ and $\bu$ off $\Gamma$.
The reason why we consider normal extensions is because they are  convenient for the error analysis.
On $\Gamma$ we consider the surface rate-of-strain tensor \cite{GurtinMurdoch75} given by
\begin{equation} \label{strain}
 E_s(\bu):= \frac12 \bP (\nabla \bu +\nabla \bu^T)\bP = \frac12(\nabla_\Gamma \bu + \nabla_\Gamma \bu^T).
 \end{equation}
We also define the surface divergence operators for a vector $\bv: \Gamma \to \R^3$ and
a tensor $\bA: \Gamma \to \mathbb{R}^{3\times 3}$:
\[
 \divG \bv  := \tr (\gradG \bv), \qquad
 \divG \bA  := \left( \divG (\be_1^T \bA),\,
               \divG (\be_2^T \bA),\,
               \divG (\be_3^T \bA)\right)^T,
               \]
with $\be_i$ the $i$th basis vector in $\R^3$.

For a given  force vector $\mathbf{f} \in L^2(\Gamma)^3$, with $\mathbf{f}\cdot\bn=0$, and
source term $g\in L^2(\Gamma)$, with $\int_\Gamma g\,ds=0$, we consider the following surface Stokes problem:
Find a vector field $\bu:\, \Gamma \to \R^3$, with $\bu\cdot\bn =0$, such that
\begin{align} 
  - \bP \divG (E_s(\bu))+\alpha \bu +\nabla_\Gamma p &=  \blf \quad \text{on}~\Gamma,  \label{strongform-1} \\
  \divG \bu & =g \quad \text{on}~\Gamma, \label{strongform-2}
\end{align}
Here $\bu$ is the tangential fluid velocity, $p$  the surface fluid pressure, and
$\alpha\ge0$  is a real parameter. The steady Stokes problem corresponds to $\alpha=0$, while $\alpha>0$ leads to a generalized Stokes problem, which results from an implicit time
integration applied to a non-stationary Stokes equation ($\alpha$ being proportional
to the inverse of the time step).
The non-zero source term $g$ in problem \eqref{strongform-1}-\eqref{strongform-2}
is included to facilitate the  treatment of evolving fluidic interfaces as a next research step. In that case,
the inextensibility condition reads $\divG \bu_T =-u_N\kappa$, where $\kappa$ is the mean curvature
and $u_N$ is the normal component of the velocity. Indeed,
we use the velocity decompostion into tangential and normal components:
\begin{equation}\label{u_T_N}
\bu = \bu_T + u_N\bn,\quad \bu_T\cdot\bn=0.
\end{equation}
We remark that further in the text we use both $u_N$
and $\bn\cdot\bu$  to denote the normal compenent of the velocity $\bu$.
For the derivation of (Navier-)Stokes equations  for evolving fluidic interfaces
see, e.g., \cite{Jankuhn1}.

From problem \eqref{strongform-1}-\eqref{strongform-2} one readily observes the following:
the pressure field is defined up a hydrostatic mode; for $\alpha=0$
all tangentially rigid surface fluid motions,  i.e. satisfying  $E_s(\bu)=0$,
are in the kernel of the differential operators at the left-hand side of eq.~\eqref{strongform-1}.
Integration by parts implies the
consistency condition for the right-hand side of eq.~\eqref{strongform-1}:
 \begin{equation}\label{constr}
 \int_\Gamma \blf \cdot \bv\,ds=0\quad\text{for all smooth tangential vector fields}~~\bv~~ \text{s.t.}~~ E_s(\bv)=\mathbf{0}.
 \end{equation}
This condition  is necessary for the well-posedness
of problem \eqref{strongform-1}-\eqref{strongform-2} when $\alpha=0$.
In the literature a tangential vector field $\bv$ defined on a surface and satisfying
$E_s(\bv)=\mathbf{0}$ is known as \textit{Killing vector field} (cf., e.g., \cite{sakai1996riemannian}).
For a smooth two-dimensional Riemannian manifold,
Killing vector fields form a Lie algebra of dimension 3 at most.
The subspace of all the Killing vector fields on $\Gamma$
plays an important role in the analysis of the problem \eqref{strongform-1}-\eqref{strongform-2}.

We introduce the following assumption:

\begin{assumption}\label{A1} We assume that either no non-trivial Killing vector field exists on
$\Gamma$ or $\alpha>0$.
\end{assumption}

\begin{remark} \rm We briefly comment on assumption~\ref{A1}. If a non-trivial Killing vector field exists on $\Gamma$,
then for the well-posedness of the surface Stokes problem \eqref{strongform-1}-\eqref{strongform-2} with $\alpha =0$
 one has to restrict the velocity space to a suitable space that does not contain these Killing fields.  If
eq.~\eqref{strongform-1} with $\alpha=0$ is understood as
an equation for the equilibrium motion, i.e. the steady state, then
the equilibrium solution is uniquely defined by an  initial velocity $\bu_0$.
In fact, all tangentially rigid modes in $\bu_0$ are conserved by the time-dependent surface Stokes equation.
In order to find the unique equilibrium motion one has to consider a
 time-discretization for the time-dependent surface Stokes equation,
that is equation~\eqref{strongform-1} with $\alpha>0$. In order to avoid these (technical)
difficulties for the case $\alpha=0$ and non-trivial Killing vector fields, we introduce assumption~\ref{A1}.
\end{remark}

\begin{remark} \label{remLaplacian} \rm
The operator $ \bP \divG E_s(\cdot)$ in equation~\eqref{strongform-1} models surface diffusion, which
is a key component in modeling Newtonian surface fluids and fluidic membranes \cite{scriven1960dynamics,GurtinMurdoch75,barrett2014stable,Gigaetal,Jankuhn1}.
In the literature there are different formulations of the surface Navier--Stokes equations,
some of which are formally obtained by substituting Cartesian differential operators by their geometric counterparts~\cite{Temam88,cao1999navier}.  These formulations may involve different surface Laplace type operators, e.g., Hodge--de Rham Laplacian.
We refer to \cite{Jankuhn1} for a brief overview of different formulations of  the surface Navier--Stokes equations.
\end{remark}
\smallskip

For the weak formulation of problem \eqref{strongform-1}-\eqref{strongform-2},
we introduce the space $\bV:=H^1(\Gamma)^3$ and norm
 \begin{equation} \label{H1norm}
  \|\bu\|_{1}^2:=\int_{\Gamma}(\|\bu(s)\|^2 + \|\nabla\bu^e (s)\|^2)\,ds,
 \end{equation}
where $\|\cdot\|$ denotes the vector $\ell^2$-norm and the matrix  Frobenius norm. We define the spaces
\begin{equation}   \label{defVT}
 \bV_T:= \{\, \bu \in \bV~|~ \bu\cdot \bn =0\,\},\quad E:= \{\, \bu \in \bV_T~|~ E_s(\bu)=\mathbf{0}\,\}.
\end{equation}
Note that $E$ is a closed subspace of $\bV_T$ and $\mbox{dim}(E)\le 3$.
We define the Hilbert space $\bV_T^0$ as an orthogonal complement of  $E$ in $\bV_T$
(hence $\bV_T^0 \sim \bV_T/E$).
For $\bu \in \bV$ we will use the orthogonal decomposition into tangential and normal parts
as in \eqref{u_T_N}.
In what follows, we will need both general and tangential vector fields on $\Gamma$.
Finally, we define $L_0^2(\Gamma):=\{\, p \in L^2(\Gamma)~|~ \int_\Gamma p\,dx=0\,\}$.

Consider the bilinear forms (with $A:B={\rm tr}\big(AB^T\big)$ for  $A,B\in\mathbb{R}^{3\times3}$)
\begin{align}
a(\bu,\bv)& := \int_\Gamma (E_s(\bu):E_s(\bv)+\alpha\bu\cdot\bv) \, ds, \quad \bu,\bv \in \bV, \label{defblfa} \\
b_T(\bu,p) &:= - \int_\Gamma p\,\divG \bu_T \, ds,  \quad \bu \in \bV, ~p \in L^2(\Gamma). \label{defblfb}
\end{align}
We emphasize that in the definition of $b_T(\bu,p)$ only the \emph{tangential} component of $\bu$ is used, i.e., $b_T(\bu,p)=b_T(\bu_T,p)$ for all $\bu \in \bV$, $p\in L^2(\Gamma)$. This property motivates the notation $b_T(\cdot,\cdot)$ instead of $b(\cdot,\cdot)$. 

The weak (variational) formulation of the surface Stokes problem \eqref{strongform-1}-\eqref{strongform-2}
reads: Determine 
$(\bu_T,p) \in \bV_T \times L_0^2(\Gamma)$ 
such that
 \begin{align}
           a(\bu_T,\bv_T) +b_T(\bv_T,p) &=(\blf,\bv_T)_{L^2} \quad \text{for all}~~\bv_T \in \bV_T, \label{Stokesweak1_1} \\
           b_T(\bu,q) & = (g,q)_{L^2} \qquad \text{for all}~~q \in L^2(\Gamma). \label{Stokesweak1_2}
 \end{align}
Here, $(\cdot,\cdot)_{L^2}$ denotes the $L^2$ scalar product on $\Gamma$.
The following surface Korn   inequality and inf-sup property were derived in \cite{Jankuhn1}.
\begin{lemma} \label{Kornlemma}
Assume $\Gamma$ is $C^2$ smooth and compact. There exist $c_K >0$ and $c_0>0$ such that
 \begin{equation} \label{korn}
\|E_s(\bv_T)\|_{L^2} \geq c_K \|\bv_T\|_{1} \quad \text{for all}~~\bv_T \in \bV_T^0,
 \end{equation}
and
\begin{equation} \label{infsup}
 \sup_{\bv_T\in{\bV_T^0}}\frac{b_T(\bv_T,p)}{\|\bv_T\|_{1}{\color{blue}}}  \geq c_0 \|p\|_{L^2} \quad \text{for all}~~p\in L^2_0(\Gamma).
\end{equation}
\end{lemma}

Since $E$ is finite dimensional (and so all norms on $E$ are equivalent), inequality \eqref{korn} implies
 \begin{equation} \label{korn1}
\|\bv_T\|_{L^2}+ \|E_s(\bv_T)\|_{L^2} \geq c_K \|\bv_T\|_{1} \quad \text{for all}~~\bv_T \in \bV_T.
 \end{equation}

Using  \eqref{korn}
and Assumption~\ref{A1}  for $\alpha=0$ and  \eqref{korn1} for $\alpha>0$ we obtain the norm equivalence
\begin{equation} \label{normeq}
  a(\bv_T,\bv_T) \simeq \|\bv_T\|_1^2 \quad \text{for all}~~\bv_T \in \bV_T.
\end{equation}
Inf-sup stability of $b_T(\cdot,\cdot)$ on $\bV_T\times L^2_0(\Gamma)$ follows from  \eqref{infsup}.
Both bilinear forms $a(\cdot,\cdot)$ and $b_T(\cdot,\cdot)$ are  continuous.
Therefore \emph{problem \eqref{Stokesweak1_1}-\eqref{Stokesweak1_2} is well posed},
and its unique solution is further denoted by $\{\bu_T^\ast,p^\ast\}$.

\section{Penalty formulation}\label{s_pen}
The weak formulation \eqref{Stokesweak1_1}-\eqref{Stokesweak1_2}  is not very suited
for a Galerkin finite element discretization, since it requires finite element functions that are \emph{tangential} to $\Gamma$.
Such functions are not easy to construct.
Thus, following \cite{hansbo2016stabilized,hansbo2016analysis,Jankuhn1,reuther2017solving}
we consider a variational formulation in a larger space $\bV_\ast \supset \bV_T$, introduced below, augmented
by a \emph{penalty term to enforce the tangential constraint weakly}. This variational method will be the basis for the finite element method introduced in section~\ref{s_TraceFEM}.

In order to write the alternative variational formulation, we introduce
the following Hilbert space and corresponding norm:
\[
\bV_\ast  :=\{\, \bu \in L^2(\Gamma)^3\,:\,\bu_T \in \bV_T,~u_N\in L^2(\Gamma)\,\},  \quad\text{with}~~
\|\bu\|_{V_\ast}^2:=\|\bu_T\|_{1}^2+\tau\|u_N\|_{L^2}^2,
\]
where $\tau$ is a positive real parameter. We also report the following useful relation:
\begin{equation} \label{idfund}
E_s(\bu)=E_s(\bu_T) + u_N \bH,
\end{equation}
where $\bH := \nabla_\Gamma \bn$ is the shape operator (second fundamental form) on $\Gamma$.
Let us define the bilinear form
\begin{equation}\label{a_aug}
{a}_\tau(\bu,\bv):= \int_\Gamma\left( E_s( \bu): E_s( \bv)+\alpha\bu_T\cdot\bv_T \right)\, ds+\tau \int_{\Gamma} u_N v_N \, ds
\end{equation}
for $\bu,\bv \in \bV_\ast$. {Using \eqref{idfund} we can rewrite it as}
\begin{equation}\label{a_auga}
\begin{split}
{a}_\tau(\bu,\bv) & = a( \bu_T,\bv_T)+ \int_\Gamma E_s( \bu_T): v_N \bH \, ds + \int_\Gamma E_s( \bv_T): u_N \bH \, ds\\ & \quad +
(\|\bH\|^2u_N,v_N)_{L^2} + \tau (u_N,v_N)_{L^2},
\end{split}
\end{equation}
{which 
is  well-defined on $\bV_\ast\times \bV_\ast$.}
In \eqref{a_aug}, $\tau$ is an \emph{augmentation} (penalty) \emph{parameter}.
Using
\[
  2\int_\Gamma E_s( \bu_T): u_N \bH \, ds \geq - \frac12 \|E_s( \bu_T)\|_{L^2}^2 - 2 \|\bH\| u_N\|_{L^2}^2
\]
and \eqref{normeq} we conclude that if $\tau \geq \max \{1, 2 \|\bH\|_{L^\infty(\Gamma)}^2\}$, there are constants $c_0 >0$, $c_1$, \emph{independent of} $\tau$ such that
\begin{equation} \label{normeq1}
 c_0 \|\bu\|_{V_\ast}^2 \leq a_\tau(\bu,\bu) \leq c_1 \|\bu\|_{V_\ast}^2 \quad \text{for all}~~\bu \in V_\ast.
\end{equation}
\begin{assumption}\label{A2} In the remainder we assume that  $\tau \geq \max \{1, 2 \|\bH\|_{L^\infty(\Gamma)}^2\}$ holds.
\end{assumption}

The alternative variational formulation reads: Find $(\hat\bu,\hat p) \in \bV_\ast \times L_0^2(\Gamma)$ such that
 \begin{align}
           {a}_\tau(\hat \bu,\bv) +b_T(\bv,\hat p) &=(\blf,\bv)_{L^2} \quad \text{for all}~~\bv \in \bV_\ast \label{Stokesweak1B_1}\\
           b_T(\hat \bu,q) & = (g,q)_{L^2}\quad \text{for all}~~ q \in L^2(\Gamma). \label{Stokesweak1B_2}
 \end{align}
Well-posedness of the \textit{augmented surface Stokes problem} \eqref{Stokesweak1B_1}-\eqref{Stokesweak1B_2}
and an estimate on the difference between its solution and the solution to \eqref{Stokesweak1_1}-\eqref{Stokesweak1_2}
are given in the following theorem, which extends a result in \cite{Jankuhn1}.
 \begin{theorem}\label{Th_aug}
Problem \eqref{Stokesweak1B_1}-\eqref{Stokesweak1B_2} is well posed.
For the unique solution $(\hat\bu, \hat{p}) \in \bV_\ast \times L_0^2(\Gamma)$ of this problem and  the  unique solution  $(\bu_T^\ast,p^\ast) \in \bV_T \times L_0^2(\Gamma)$ of  \eqref{Stokesweak1_1}-\eqref{Stokesweak1_2} the following estimate holds
 \begin{equation}\label{EstAug}
 \|\hat\bu_T-\bu_T^\ast\|_{1}+\|\hat u_N\|_{L^2}+\|\hat p-p^\ast\|_{L^2}\le C\,\tau^{-1}(\|\blf\|_{L^2}+\|g\|_{L^2}),
 \end{equation}
 where $C$ depends only on $\Gamma$.
 \end{theorem}

\begin{proof}
The bilinear form $a_\tau(\cdot,\cdot)$ is continuous and elliptic on $\bV_\ast$, as shown in~\eqref{normeq1}.
 The uniform in $\tau $ inf-sup property and continuity for $b_T(\cdot,\cdot)$ on
 $\bV_\ast \times L_0^2(\Gamma)$ immediately follow from \eqref{infsup}, the embedding $ \bV_T^0\subset \bV_\ast$ and the property $b_T(\bv,q)=b_T(\bv_T,q)$.
 Therefore,  problem \eqref{Stokesweak1B_1}-\eqref{Stokesweak1B_2}
 is well posed and the following a priori estimate holds
 \begin{equation}\label{aux1b}
 \|\hat\bu\|_{V_\ast}+\|\hat p\|_{L^2} \le c(\|\blf\|_{\bV_\ast'}+\|g\|_{L^2}) \le c(\|\blf\|_{L^2}+\|g\|_{L^2}),
 \end{equation}
 with some $c$ independent of  $\tau$ and $\blf,g$.

 We  test equation~\eqref{Stokesweak1B_1} with $\bv=\hat u_N\bn$. Thanks to \eqref{idfund} and $\blf \cdot \bn=0$, we obtain the identity
 \[
 \int_\Gamma  E_s( \hat\bu):\hat u_N \bH ds+\tau\|\hat u_N\|^2_{L^2}=0.
 \]
 Then, the Cauchy inequality and inequality \eqref{aux1b} lead to
\begin{equation}\label{aux1c}
\begin{split}
 \tau\|\hat u_N\|^2_{L^2}&=-\int_\Gamma  E_s( \hat\bu):\hat u_N \bH ds\le C\,\|\hat u_N\|_{L^2}\|E_s( \hat\bu)\|_{L^2}\le
 C\,\|\hat u_N\|_{L^2}\|\hat\bu\|_{V_\ast}\\& \le  C\,\|\hat u_N\|_{L^2}(\|\blf\|_{L^2}+\|g\|_{L^2}).
\end{split}
 \end{equation}
Hence, we proved the desired estimate for $ \|\hat u_N\|_{L^2}$. We now consider the term $ \|\hat\bu_T-\bu_T^\ast\|_{1}$. We take  $\bv_T:=\hat\bu_T-\bu_T^\ast$ in equations~\eqref{Stokesweak1_1}-\eqref{Stokesweak1_2}
and \eqref{Stokesweak1B_1}-\eqref{Stokesweak1B_2}.
From the divergence equations in \eqref{Stokesweak1_2} and \eqref{Stokesweak1B_2} we obtain $b_T(\bv_T,q)=0$ for all $q \in L^2(\Gamma)$.  Taking $\bv=\bv_T$ in \eqref{Stokesweak1_1} and \eqref{Stokesweak1B_1} and using $b_T(\bv_T,q)=0$ for all $q \in L^2(\Gamma)$ we get
\[
  a(\bu_T^\ast,\bv_T)-a_\tau(\hat \bu, \bv_T)=0,
\]
and using {\eqref{a_auga}} we obtain
\[
  a(\bv_T,\bv_T)= -\int_\Gamma E_s(\bv_T) : \hat u_N \bH \, ds.
\]
From this and \eqref{normeq} we conclude
 \begin{align}
  \|\bu_T^\ast-\hat \bu_T\|_{1}^2 &\le c\, a(\bv_T,\bv_T)
 \le c \|\bu_T^\ast-\hat \bu_T\|_{1} \|\hat{u}_N\|_{L^2} . \label{est2}
 \end{align}
Hence, $ \|\bu_T^\ast-\hat \bu_T\|_{1} \leq c  \|\hat{u}_N\|_{L^2} \leq c \tau^{-1}(\|\blf\|_{L^2}+\|g\|_{L^2})$ holds, which is the desired estimate for $ \|\hat\bu_T-\bu_T^\ast\|_{1}$.
Finally, the estimate for $\|\hat p-p^\ast\|_{L^2}$ follows from inequalities
\eqref{infsup}, \eqref{aux1c}, and \eqref{est2}:
 \begin{equation*}
 \begin{split}
c\,\|p^\ast-\hat p\|_{L^2}&\le \sup_{\bv_T\in{\bV_T^0}}\frac{b_T(\bv_T,p^\ast- \hat p)}{\|\bv_T\|_1}= \sup_{\bv_T\in{\bV_T^0}} \frac{a(\bu_T^\ast,\bv_T)-a_\tau(\hat \bu,\bv_T)}{\|\bv_T\|_1}\\
&= \sup_{\bv_T\in{\bV_T^0}}\frac{a(\bu_T^\ast-\hat\bu_T,\bv_T)- \int_\Gamma E_s( \bv_T): \hat u_N \bH\,ds}{\|\bv_T\|_1} \\&\le C(\|\hat\bu_T-\bu_T^\ast\|_{1} +\|\hat{u}_N\|_{L^2})\le C\,\tau^{-1}(\|\blf\|_{L^2}+\|g\|_{L^2}).
\end{split}
\end{equation*}
This concludes the proof.
\end{proof}

\medskip

If one is interested in a finite element method of order $m$ for the surface Stokes problem
\eqref{strongform-1}-\eqref{strongform-2},
then Theorem~\ref{Th_aug} suggests to use weak formulation
\eqref{Stokesweak1B_1}-\eqref{Stokesweak1B_2} with penalty parameter $\tau=O(h^{-m})$.
For the particular choice of $P_1$--$P_1$ elements used in this paper, this motivates
\begin{equation}\label{tau_h}
\tau= c_\tau h^{-2}
\end{equation}
with some $c_\tau$ depending only on $\Gamma$.

We conclude by stressing again that the weak formulation \eqref{Stokesweak1B_1}-\eqref{Stokesweak1B_2}
gives a numerical advantage over formulation \eqref{Stokesweak1_1}-\eqref{Stokesweak1_2}
by not forcing the use of tangential finite element vector fields.

\section{Trace Finite Element Method}\label{s_TraceFEM}
For the discretization of the variational problem \eqref{Stokesweak1B_1}-\eqref{Stokesweak1B_2}
we extend the trace finite element approach (TraceFEM) introduced in \cite{ORG09}
for elliptic equations on surfaces. In this section, we
present and analyze the method.

Let $\Omega \subset \R^3$ be a fixed polygonal domain  that strictly contains $\Gamma$.
We consider a family of shape regular tetrahedral triangulations $\{\T_h\}_{h >0}$ of $\Omega$.
The subset of tetrahedra that have a nonzero intersection with $\Gamma$ is collected in the set
denoted by $\T_h^\Gamma$. For the analysis of the method,  we assume $\{\T_h^\Gamma\}_{h >0}$ to be quasi-uniform.
However, in practice adaptive mesh refinement is possible, as discussed, for example, in \cite{DemlowOlsh,chernyshenko2015adaptive}.
The domain formed by all tetrahedra in $\T_h^\Gamma$ is denoted by $\OGamma:=\text{int}(\overline{\cup_{T \in \T_h^\Gamma} T})$.
On $\T_h^\Gamma$ we use a standard finite element space of continuous functions that are piecewise-affine functions.
In this paper we focus on $P_1$ elements, i.e. polynomials of degree $1$.
This so-called \emph{bulk finite element space} is denoted by $V_h$.

Since the $P_1$--$P_1$ pair for velocity and pressure is not inf-sup stable,
a stabilization term is added to the finite element (FE) formulation (see below).
For this extra term, we need an extension of the normal vector field $\bn$
from $\Gamma$ to $\OGamma$, denoted with $\bn^e$.
We choose $\bn^e = \nabla d$, where $d$ is the signed distance function to $\Gamma$.
In practice, $d$ is often not available and thus we use approximations,
as discussed in Remark~\ref{remimplementation}.
Another implementation aspect of TraceFEM that requires attention is the computation of integrals over the surface $\Gamma$ with sufficiently high  accuracy.
 In practice, $\Gamma$ can be defined implicitly as the zero level
of a level set function and a parametrization of $\Gamma$ may not be available.
{An easy way to compute approximation $\Gamma_h \approx \Gamma$ and the corresponding geometric errors} will also be discussed in Remark~\ref{remimplementation}. Below
we use the exact extended normal $\bn=\bn^e$ and we \emph{assume exact integration over} $\Gamma$.

Consider  the spaces
\[ \bU:=\{\, \bv \in H^1(\OGamma)^3~|~\bv|_{\Gamma} \in \bV_\ast\,\}, \quad Q:=H^1(\OGamma).
\]
Our velocity and pressure finite element spaces are $P_1$ continuous FE spaces on $\OGamma$:
\[
\bU_h:= (V_h)^3\subset \bU, \quad Q_h :=V_h\cap L^0_2(\Gamma)\subset Q.
\]
We introduce the following finite element bilinear forms:
\begin{align}
 A_h(\bu,\bv) &:= a_\tau(\bu,\bv) + \rho_u \int_{\OGamma} (\nabla \bu \bn)\cdot (\nabla  \bv \bn) \, dx, \qquad \bu,\bv \in \bU, \label{Ah} \\
s_h(p,q)& := \rho_p  \int_{\OGamma} \nabla p\cdot\nabla q \, dx,     \qquad p,q \in Q. \label{sh}
\end{align}
The volumetric term in the definition of $A_h$ is the so called \emph{volume normal derivative} stabilization
first introduced  in \cite{burman2016cutb,grande2017higher} in the context of  TraceFEM
for the scalar Laplace--Beltrami problem on a surface. The term vanishes for the strong solution $\bu$ of
equations~\eqref{strongform-1}-\eqref{strongform-2}, since one can always assume a normal extension of $\bu$ off the surface.
The purpose of this additional term is to stabilize the \emph{resulting algebraic system} against possible
instabilities produced by the small cuts of the background triangulation by the surface.
Indeed, if one uses a natural nodal basis in $V_h$, then small cuts of background tetrahedra may lead to
(arbitrarily) small diagonal entries in the resulting matrices.
The stabilization term in \eqref{Ah} eliminates this problem because it
allows to get control over the $L^2(\OGamma)$-norm of $\bv_h\in\bU_h$
by the problem induced norm $A_h(\bv_h,\bv_h)^{\frac12}$ for a suitable choice of $\rho_u$. We note that other efficient stabilization techniques exist; see~\cite{burman2016cutb} and  the review in~\cite{olshanskii2016trace}.

The role of $s_h$ defined in \eqref{sh} is twofold.
First, it stabilizes the nodal basis in the pressure space $Q_h$ with respect to small element cuts,
in the same way as the volumetric term in \eqref{Ah} does this for velocity.
Then, it stabilizes the velocity--pressure pair against the violation of the inf-sup condition (the discrete counterpart of \eqref{infsup}). For the latter, the $s_h$ stabilization resembles
the well-known  Brezzi--Pitk\"{a}ranta stabilization ~\cite{BP} for the planar Stokes $P_1$--$P_1$ finite elements.
Both roles are clearly seen from the decomposition:
\begin{equation}\label{form_s}
s_h(p,q) =  \underbrace{\rho_p\int_{\OGamma}  \frac{\partial p}{\partial\bn} \frac{\partial q}{\partial\bn} \, dx}_{\text{normal stabilization}}~+  \underbrace{\rho_p  \int_{\OGamma} \nablaG p\,\nablaG q \, dx}_{\text{Brezzi--Pitk\"{a}ranta stabilization}}.
\end{equation}

The analysis of the scalar Laplace--Beltrami surface equation~\cite{burman2016cutb,grande2017higher}
and vector surface Laplacians~\cite{gross2017trace} suggest that optimal convergence and algebraic stability
should be expected
for a wide range of the normal stabilization parameter,
$h \lesssim \rho_u \lesssim h^{-1}$.
In this paper, we introduce the minimal suitable stabilization and set $\rho_p\simeq\rho_u\simeq h$.
Here and further in the paper we write $x\lesssim y$ to state that the inequality  $x\le c y$
holds for quantities $x,y$ with a constant $c$, which is independent of the mesh parameter
$h$ and the position of $\Gamma$ over the background mesh.
Similarly for $x\gtrsim y$, and $x\simeq y$
will mean that both $x\lesssim y$ and $x\gtrsim y$ hold.
Note that a $\rho_p\simeq h$ scaling is consistent with the well-known $O(h^2)$ choice of the Brezzi--Pitk\"{a}ranta stabilization parameter in the usual (planar or volumetric) case. The additional $O(h)$ scaling comes from the fact that the second term in \eqref{form_s} is computed over the narrow volumetric domain rather than over the surface.

\smallskip

The trace finite element method (TraceFEM) we use reads as follows:
Find $(\bu_h, p_h) \in \bU_h \times Q_h$ such that
\begin{equation} \label{discrete}
 \begin{aligned}
  A_h(\bu_h,\bv_h) + b_T(\bv_h,p_h) & =(\blf,\bv_h)_{L^2} &\quad &\text{for all } \bv_h \in \bU_h \\
  b_T(\bu_h,q_h)-s_h(p_h,q_h) & = (g,q_h)_{L^2} &\quad &\text{for all }q_h \in Q_h,
 \end{aligned}
\end{equation}
with the following setting for the parameters:
\begin{equation} \label{parameters}
\tau= c_\tau h^{-2},\quad \rho_p=c_p h, \quad \rho_u=c_u h.
\end{equation}
Here $h$ is the characteristic mesh size of the background tetrahedral mesh, while
$c_\tau$, $c_p$, $c_u$ are some $O(1)$ tunable constants. The optimal value
of those constants may depend on problem data such as $\Gamma$, but is
independent of $h$ and of how $\Gamma$ cuts through the background mesh.
The decomposition \eqref{form_s} suggests that one can split $s_h(\cdot,\cdot)$ into
two parts and use different scalings for the normal and tangential terms.
For simplicity of the method, we avoid this option.

\begin{remark} \label{remimplementation}
 \rm We discuss some implementation aspects of the trace finite element discretization \eqref{discrete}. In the bilinear form $A_h(\bu_h,\bv_h)$ only full gradients (no tangential ones) of the arguments are needed.  These can be computed as in standard finite element methods. It is important for the implementation that in $A_h(\cdot,\cdot)$ we do \emph{not} need derivatives of projected velocities, e.g. of $(\bu_h)_T$. In the bilinear form $b_T(\bv_h,p_h) = - \int_\Gamma p_h\,\divG (\bv_h)_T \, ds$, however, derivatives of the tangential velocity $(\bv_h)_T= \bP\bv_h$, with $\bP:=\bI- \bn \bn^T$, appear. This requires a differentation of $\bP$ and thus a sufficiently accurate curvature approximation. In a setting of $H^1$-conforming pressure finite element spaces, as used in this paper, it is convenient to rewrite the bilinear form as $b_T(\bv_h,p_h) =  \int_\Gamma \nabla_\Gamma p_h \cdot \bv_h \, ds= \int_\Gamma (\bP \nabla p_h )\cdot \bv_h \, ds  $. Implementation then only requires an approximation of
$\bn_h \approx \bn$ and not of derivatives of $\bn$. \\
As noted above, in the implementation of this method one typically replaces $\Gamma$
 by an approximation $\Gamma_h \approx \Gamma$ such that integrals over $\Gamma_h$
 can be efficiently computed. Furthermore, the exact normal $\bn$ is approximated by  $\bn_h \approx \bn $.
 In the literature on finite element methods for surface PDEs, this is standard practice.
 We will use a piecewise planar surface approximation $\Gamma_h$ with ${\rm dist}(\Gamma,\Gamma_h) \lesssim h^2$.
 If one is interested in surface FEM with higher order surface approximation, we refer to the recent paper \cite{grande2017higher}.
We assume a level set representation of $\Gamma$:
 \[
 \Gamma=\{\bx\in\mathbb{R}^3\,:\,\phi(\bx)=0\},
 \]
with some smooth function $\phi$ such that $|\nabla\phi|\ge c_0>0$ in a neighborhood of $\Gamma$.
 For the numerical experiments in section~\ref{sectExp} we use a piecewise planar surface approximation:
\[
\Gamma_h=\{\bx\in\mathbb{R}^3\,:\, I_h(\phi(\bx))=0\},
\]
where $I_h(\phi(\bx))\in V_h$ is the nodal interpolant of $\phi$.
 As for the  construction of suitable normal approximations $\bn_h \approx \bn$,
 several techniques are available in the literature.
 One possibility is to use $\bn_h(\bx)=\frac{\nabla \phi_h(\bx)}{\|\nabla \phi_h(\bx)\|_2}$,
 where $\phi_h$ is a finite element approximation of a level set function $\phi$ which characterizes $\Gamma$.
 This is technique we use in section~\ref{sectExp}, where $\phi_h$ is defined as a $P_2$ nodal interpolant of $\phi$.
  Analyzing the effect of such geometric errors is beyond the scope of this paper.
Our focus is on analyzing the TraceFEM \eqref{discrete}.
\end{remark}

\begin{remark} \rm An alternative numerical approach to enforce the tangentiality constraint on the flow field $\bu$ is to introduce a Lagrange multiplier in \eqref{discrete} instead of using a  penalty approach. This adds extra Lagrange multiplier unknowns to the algebraic system, but removes  the augmentation parameter $\tau$. For a surface vector-Laplace problem  the TraceFEM with such an enforcement of $\bu\cdot\bn=0$  was studied in \cite{gross2017trace}. For $P_1$ velocity elements one can use a $P_1$ Lagrange multiplier space for a numerically stable and optimally accurate TraceFEM. {A systematic comparison of the penalty approach presented in this paper and such a Lagrange multiplier technique for the surface Stokes problem is a topic for future research.}
\end{remark}

\section{Error analysis of TraceFEM}\label{s_error}
In this section we present stability and error analysis of the finite element method \eqref{discrete}.
After some preliminaries we derive a discrete inf-sup result and discuss the consistency between the FE formulation and the original problem. We then prove an $O(h)$ error  estimate in the natural energy norm and an $O(h^2)$ error estimate in the surface $L^2$-norm for velocity.
\subsection{Preliminaries}
In this section we collect a few results that we need in the error analysis. The parameters
in the bilinear forms are set as in \eqref{parameters}.
We introduce the norms:
\begin{align}
 \|\bv\|_U^2 & :=A_h(\bv,\bv), \quad \bv \in \bU, \label{defUh} \\
  \|q\|_Q^2 &:= \|q\|_{L^2}^2 + h \|\nabla q\|_{L^2(\OGamma)}^2, \quad q \in Q. \label{defmuast}
\end{align}
In these norms we easily obtain  continuity estimates.
The  Cauchy-Schwarz inequality and the definition of the norms immediately yield the following estimates:
\begin{align}
 A_h(\bu,\bv) & \leq \|\bu\|_U\|\bv\|_U \quad \text{for all }\bu,\bv \in \bU, \label{cont1}\\
 b_T(\bu,q) & \lesssim  \|\bu\|_U \|q\|_Q \quad \text{for all }\bu \in \bU, q \in Q. \label{cont2}\\
 s_h(p,q) & \lesssim  \|p\|_Q \|q\|_Q \quad \text{for all } p,\,q \in Q. \label{cont3}
\end{align}
On the finite element spaces $\bU_h$, $Q_h$ the norms $\|\cdot\|_U$ and $\|\cdot\|_Q$ are \emph{uniformly} equivalent to certain (scaled) $L^2$ and $H^1$ norms. The uniformity holds with respect to $h$ and the position of $\Gamma$ in the background mesh. We recall a result known from the literature
for scalar finite element function $v_h \in V_h$:
 \begin{equation} \label{fund1}
  h \|v_h\|_{L^2}^2 + h^2 \|\bn \cdot \nabla v_h\|_{L^2(\OGamma)}^2 \simeq \|v_h\|_{L^2(\OGamma)}^2\quad \text{for all}~~v_h \in V_h.
\end{equation}
A proof of the $\gtrsim$ estimate in \eqref{fund1} is given in \cite{grande2017higher},
while the $\lesssim$ estimate follows from the inequality, cf.~\cite{Hansbo02}:
\begin{equation} \label{fund1B}
 h \|v\|_{L^2(\Gamma\cap K)}^2 \lesssim  \|v\|_{L^2(K)}^2+h^2\|v\|_{H^1(K)}^2 \quad \text{for all}~v \in H^1(K),~K\in\T_h^\Gamma,
\end{equation}
and a standard finite element inverse inequality $\|v_h\|_{H^1(\OGamma)} \lesssim h^{-1}\|v_h\|_{L^2(\OGamma)}$ for all $v_h \in V_h$.
Using \eqref{fund1}, \eqref{normeq1} and \eqref{parameters} we get  for vector functions $\bv=(v_1,v_2,v_3) \in \bU_h$:
\begin{align*}
 A_h(\bv,\bv) & \simeq \|\bv\|_{V_\ast}^2 + h \sum_{i=1}^3 \|\bn \cdot \nabla v_i\|_{L^2(\OGamma)}^2 \\
  & \simeq \|\bv_T\|_1^2 + h^{-2}\|v_N\|_{L^2}^2 + \sum_{i=1}^3 \big( \|v_i\|_{L^2}^2 + h \|\bn \cdot \nabla v_i\|_{L^2(\OGamma)}^2\big ) \\
 & \simeq \|\bv_T\|_1^2 + h^{-2}\|v_N\|_{L^2}^2 + h^{-1} \|\bv\|_{L^2(\OGamma)}^2,
\end{align*}
i.e.:
\begin{equation} \label{equiv1}
\|\bv\|_U^2\simeq \|\bv_T\|_1^2 +h^{-2}\|v_N\|_{L^2}^2 + h^{-1} \|\bv\|_{L^2(\OGamma)}^2 \quad  \text{for all}~~\bv\in\bU_h.
\end{equation}
From  \eqref{fund1} and finite element inverse inequalities  we also obtain:
\begin{equation} \label{fundC}
\|q\|_Q\simeq  h^{-\frac12} \|q\|_{L^2(\OGamma)}\quad \text{for all}~~ q \in Q_h.
\end{equation}

\subsection{Discrete inf-sup property}
Based on the inf-sup property \eqref{infsup} for the continuous problem and a perturbation argument we derive a discrete inf-sup result. In the analysis, for a scalar function $v \in L^2(\Gamma)$ we use a \emph{constant extension along normals} denoted by $v^e$, which is defined in a fixed sufficiently small neighborhood of $\Gamma$ that contains (for $h$ sufficiently small) the local triangulation $\OGamma$,  cf.~\cite{DEreview}.
For $v^e$ the following estimates hold \cite{reusken2015analysis}:
 \begin{equation}\label{normal}
\begin{split}
 h^\frac{1}{2}\|\nablaG v\|_{L^2}& \simeq \|\nabla v^e\|_{L^2(\OGamma)}, \quad \text{for all}~~v \in H^1(\Gamma), \\
h^{\frac12}\|v\|_{L^2}& \simeq \|v^e\|_{L^2(\OGamma)}, \quad \text{for all}~~v \in L^2(\Gamma),\\
  \|v^e\|_{H^2(\OGamma)} &\lesssim h^{\frac12} \|v\|_{H^2(\Gamma)}, \quad \text{for all}~~v \in H^2(\Gamma).
\end{split}
 \end{equation}
The componentwise constant extension along normals of a vector function $\bv$ is denoted by $\bv^e$.
Applying the first estimate in \eqref{normal} componentwise we also get for all $\bv \in H^1(\Gamma)^3$:
  \begin{equation}\label{normalv}
h^{\frac12}\|\nabla\bv\bP\|_{L^2}  \simeq \|\nabla\bv^e\|_{L^2(\OGamma)}, \text{ and thus}\quad \|\nabla\bv^e\|_{L^2(\OGamma)}\lesssim h^{\frac12}\|\bv\|_1.
 \end{equation}
\begin{lemma} \label{lemdiscreteinfsup}
If the  constant $c_p$ in \eqref{parameters} is taken sufficiently large (independent of $h$ and of how $\Gamma$ intersects the background mesh) then the following holds:
 \begin{equation} \label{infsupest}
  \sup_{\bv_h \in \bU_h} \frac{b_T(\bv_h, q_h)}{\|\bv_h\|_U} + s_h(q_h,q_h)^\frac12 \gtrsim\| q_h\|_Q \quad \text{for all}~~ q_h \in Q_h.
 \end{equation}
\end{lemma}

\begin{proof} Take $ q_h \in Q_h$. Thanks to the inf-sup property \eqref{infsup}, there exists
$\bv\in\bV_T$ such that
\begin{equation} \label{aux10}
  b_T( \bv,  q_h)= \| q_h\|_{L^2}^2,\qquad c_0\|\bv\|_1\le \| q_h\|_{L^2}.
\end{equation}
Let  $\bv^e$ be the normal  extension of $\bv$ and take  $\bv_h:= I_h ( \bv^e) \in \bU_h$, where $I_h: H^1(\mathcal{O}(\Gamma))^3\to \bU_h$ is the Cl\'ement interpolation operator, with $\mathcal{O}(\Gamma)$  a  neighborhood of $\Gamma$ that contains $\OGamma$ and of width $\mathcal{O}(h)$.
Based on \eqref{fund1B}, approximation properties of $I_h ( \bv^e)$, and \eqref{normal}--\eqref{normalv} one gets
by standard arguments (see, e.g., \cite{reusken2015analysis}):
\begin{equation}\label{aux1d}
 \|\bv - I_h ( \bv^e)\|_{L^2}+h\|\nablaG(\bv-I_h ( \bv^e))\|_{L^2}\lesssim h\|\nablaG\bv\|_{L^2}.
\end{equation}

Due to \eqref{equiv1}, \eqref{aux1d}, $v_N=0$ (since $\bv\in\bV_T$), \eqref{fund1B}, and \eqref{normalv}, we have
\begin{equation*}
\begin{split}
 \|\bv_h\|_U&=\|I_h ( \bv^e)\|_U\\
{\footnotesize \eqref{equiv1}}~ & \simeq \|I_h ( \bv^e)_T\|_1+h^{-1}\|I_h ( \bv^e)_N\|_{L^2}+ h^{-\frac12} \|I_h ( \bv^e)\|_{L^2(\OGamma)}\\
{\footnotesize v_N=0,\,\eqref{fund1B}}~  &\lesssim h^{-\frac12}\|I_h ( \bv^e)\|_{H^1(\OGamma)}+h^{-1}\|I_h ( \bv^e)-\bv)\cdot\bn\|_{L^2}\\
    & \lesssim h^{-\frac12}\|\bv^e\|_{H^1(\mathcal{O}(\Gamma))} + h^{-1} \|I_h ( \bv^e)-\bv^e\|_{L^2} \\
{\footnotesize \eqref{fund1B}}~  &\lesssim h^{-\frac12}\|\bv^e\|_{H^1(\mathcal{O}(\Gamma))}+h^{-\frac32}\|I_h ( \bv^e)-\bv^e\|_{L^2(\OGamma)} \\ & \qquad +
 h^{-\frac12}\|I_h ( \bv^e)-\bv^e\|_{H^1(\OGamma)}\\
{\footnotesize \eqref{aux1d}}~ &\lesssim h^{-\frac12}\|\bv^e\|_{H^1(\mathcal{O}(\Gamma))}\\
{\footnotesize \eqref{normalv}}~ &\lesssim \|\bv\|_1.
 \end{split}
\end{equation*}
Hence, we proved
\begin{equation} \label{aux1U}
 \|\bv_h\|_U\lesssim \|\bv\|_1.
\end{equation}

Now note, that due to \eqref{aux1d}, \eqref{aux10}, \eqref{normal},  we have
\begin{equation*} 
\begin{split}
b_T(\bv_h,  q_h)  & = b_T( \bv,  q_h)- b_T(\bv - I_h ( \bv^e), q_h) \\
&=\| q_h\|_{L^2}^2 - b_T(\bv - I_h ( \bv^e), q_h) \\
&=\| q_h\|_{L^2}^2 + (\bv - I_h ( \bv^e), \gradG q_h)_{L^2}\\
&\ge  \| q_h\|_{L^2}^2 - \|\bv - I_h ( \bv^e)\|_{L^2} \|\gradG q_h\|_{L^2}\\
&\ge  \| q_h\|_{L^2}^2 - c\,h \|\bv\|_1 \|\gradG q_h\|_{L^2}\\
& \ge  \| q_h\|_{L^2}^2 - c\,h^{\frac12} \|\bv\|_1 \|\nabla q_h\|_{L^2(\OGamma)}.
\end{split}
\end{equation*}
Dividing both sides by $\|\bv\|_1$ and using the results in \eqref{aux10} and \eqref{aux1U} we get, for suitable constants $\tilde c_0 >0$ and $\tilde c$:
\[
  \frac{b_T(\bv_h,  q_h) }{\|\bv_h\|_U} \geq\tilde c_0   \|q_h\|_{L^2}- \tilde c h^\frac12 \|\nabla q_h\|_{L^2(\OGamma)}.
\]
Hence, for $c_p \geq \tilde c^2$ we obtain the estimate \eqref{infsupest}.
\end{proof}
\ \\
\begin{assumption}\label{A3} In the remainder we assume that $c_p$ in \eqref{parameters} is taken sufficiently large such that the discrete inf-sup estimate \eqref{infsupest}  holds.
\end{assumption}

\subsection{Consistency}
For the error analysis it is convenient to introduce the bilinear form
\begin{equation} \label{defk}
 \bA_h\big((\bu,p),(\bv,q)\big):= A_h(\bu,\bv)+ b_T(\bv,p) +b_T(\bu,q) - s_h(p,q),
\end{equation}
for $(\bu,p), (\bv,q) \in \bU \times Q$. Then,
the discrete problem \eqref{discrete} has the compact representation:
Determine $(\bu_h,p_h) \in \bU_h \times Q_h$ such that
\begin{equation} \label{discrete1}
 \bA_h\big((\bu_h,p_p),(\bv_h,q_h)\big)=(\blf,\bv_h)_{L^2} + (g,q_h)_{L^2} \quad \text{for all}~~(\bv_h,q_h) \in \bU_h \times Q_h.
\end{equation}
Due to \eqref{infsupest} discrete problem \eqref{discrete1} has a unique solution, which is denoted by $(\bu_h,p_h)$.
Below we derive a consistency relation of the unique solution $(\bu_T^\ast, p^\ast)$ of \eqref{Stokesweak1_1}-\eqref{Stokesweak1_2}. To this purpose, we need the normal extensions
$(\bu_T^\ast)^e$ and  $(p^\ast)^e$ of $\bu_T^\ast$ and $ p^\ast$, respectively. To simplify the notation these extensions are also denoted by $\bu_T^\ast$ and  $p^\ast$.
\begin{lemma} \label{lemconsistency}
Let  $(\bu_T^\ast, p^\ast)$ be the unique solution of \eqref{Stokesweak1_1}-\eqref{Stokesweak1_2} and $(\bu_h,p_h)$ the unique solution of \eqref{discrete1}. The following relations hold:
\begin{equation}
 \bA_h\big((\bu_T^\ast,p^\ast)),(\bv,q)\big)= (\blf,\bv)_{L^2}+(g,q)_{L^2} + \int_{\Gamma} E_s(\bu_T^\ast):v_N \bH \, ds - s_h(p^\ast,q) \label{consis1}
\end{equation}  for all $(\bv,q) \in \bU \times Q $;
\begin{equation}\label{consis2}
\bA_h\big((\bu_T^\ast - \bu_h,p^\ast- p_h)),(\bv_h,q_h)\big) = \int_{\Gamma} E_s(\bu_T^\ast):(\bn \cdot \bv_h) \bH \, ds - s_h(p^\ast,q_h)
\end{equation}
for all $(\bv_h,q_h) \in \bU_h \times Q_h$.
\end{lemma}
\begin{proof}
 Using $\nabla \bu_T^\ast\bn = \nabla (\bu_T^\ast)^e\bn=0$ and $u_N^\ast=0$ we get
\[
 A_h(\bu_T^\ast,\bv) = a_\tau(\bu_T^\ast,\bv)= a(\bu_T^\ast, \bv_T) + \int_\Gamma E_s(\bu_T^\ast):v_N \bH \, ds.
\]
Combining this with $b_T(\bv,p)=b_T(\bv_T,p)$, $\blf \cdot \bn =0$,  and \eqref{Stokesweak1_1}-\eqref{Stokesweak1_2} we obtain
\[
 A_h(\bu_T^\ast,\bv) + b_T(\bv,p^\ast)+b_T(\bu_T^\ast,q)= (\blf,\bv)_{L^2}+(g,q)_{L^2}+ \int_\Gamma E_s(\bu_T^\ast):v_N \bH \, ds
\]
for all $(\bv,q)\in U \times Q$, which using definition \eqref{defk} yields \eqref{consis1}. The relation in \eqref{consis2} directly folows from \eqref{consis1} and \eqref{discrete1}.
\end{proof}
\ \\[1ex]
The consistency relation \eqref{consis2} describes the violation of the Galerkin orthogonality, due to the stabilization term $s_h(\cdot,\cdot)$ and due to the fact that  the finite element space contains no-tangential test functions. Below we  derive bounds for the consistency error terms in the right-hand side of \eqref{consis2}.

\subsection{Discretization  error bound} We apply the standard theory of saddle point problems to
derive the error estimates in the energy norm, defined by
\[
\|\bu,p\|:=(\|\bu\|^2_\bU+\|p\|^2_Q)^\frac12.
\]
Using \eqref{cont1}-\eqref{cont3} one easily checks  that $\bA_h(\cdot,\cdot)$ is  continuous on $\bU \times Q$  with respect to this product norm.
From Lemma~\ref{lemdiscreteinfsup} and definition \eqref{defk} we get the following inf-sup stability result:
 \begin{equation} \label{infsupA}
 0<c_0\le\inf_{(\bv_h,q_h) \in \bU_h \times Q_h}\sup_{(\bv_h,q_h) \in \bU_h \times Q_h} \frac{\bA_h\big((\bu_h,p_h),(\bv_h,q_h)\big)}{\|\bu_h,p_h\|\|\bv_h,q_h\|}.
\end{equation}
The proof of \eqref{infsupA}  for $s_h(\cdot,\cdot)=0$ is given in many finite element textbooks, e.g.~\cite{ern2013theory}. The arguments have straightfoward extensions to the case $s_h\neq0$, cf., for example,  \cite{guzman2016inf}. From the discrete inf-sup property of the $\bA_h$ bilinear form and continuity we conclude  well-posedness and a stability bound:
\begin{equation} \label{bounh1}
\|\bu_h,p_h\|\lesssim \|\blf\|_{L^2}+\|g\|_{L^2}.
\end{equation}
Furthermore, we obtain the following optimal discretization error bound.
\begin{theorem} \label{thm2}  Let $(\bu_T^\ast ,p^\ast )$ be the solution of \eqref{Stokesweak1_1}-\eqref{Stokesweak1_2} and assume that $(\bu_T^\ast ,p^\ast ) \in H^2(\Gamma) \times H^1(\Gamma)$. Let  $(\bu_h,p_h) \in \bU_h\times Q_h$ be the solution  of \eqref{discrete}. The following discretization error bounds hold:
\begin{equation}\label{discrbound}
 \|\bu_T^\ast  -\bu_h,p^\ast  - p_h \|  \lesssim h (\|\bu_T^\ast \|_{2} +\|p^\ast \|_{1}).
\end{equation}
Here $\|\cdot\|_2$ denotes the $H^2(\Gamma)$ Sobolev norm.
\end{theorem}
\begin{proof}
Using the stability and consistency properties in \eqref{infsupA} and \eqref{consis2},  we obtain, for arbitrary $(\bw_h,\xi_h) \in \bU_h \times Q_h$:
\begin{align*}
 & \|\bu_h- \bw_h,p_h- \xi_h\| \lesssim \sup_{(\bv_h,q_h)\in \bU_h \times Q_h} \frac{\bA_h\big( (\bu_h- \bw_h,p_h- \xi_h),(\bv_h,q_h)\big)}{\|\bv_h,q_h\|} \\ & = \sup_{(\bv_h,q_h)\in \bU_h \times Q_h} \frac{ \bA_h\big( (\bu_T^\ast- \bw_h,p^\ast- \xi_h),(\bv_h,q_h)\big)-s_h(p^\ast,q_h)+\int_\Gamma E_s( \bu_T^\ast):(\bn\cdot\bv_h) \bH \, ds}{\|\bv_h,q_h\|}\\ &\lesssim \|\bu_T^\ast - \bw_h,p^\ast - \xi_h\| + h^\frac{1}{2}\|\nabla p^\ast\|_{L^2(\OGamma)}+\tau^{-\frac12}\|E_s( \bu_T^\ast)\|_{L^2}.
\end{align*}
Hence, using the triangle inequality and $\tau=O(h^{-2})$ we get the error bound
\begin{equation} \label{discrerror}
\begin{split}
& \|\bu_T^\ast -\bu_h,p^\ast - p_h \| \\
  & \lesssim
 \inf_{(\bv_h,q_h)\in \bU_h \times Q_h}\big(\|\bu_T^\ast -\bv_h,p^\ast - q_h \|\big)+ h^\frac{1}{2}\|\nabla p^\ast\|_{L^2(\OGamma)}+h\|\bu_T^\ast\|_1.
\end{split}
\end{equation}
Thanks to the norm equivalences in \eqref{normal} (recall that $\nabla p^\ast= \nabla (p^\ast)^e$), we have
\begin{equation} \label{discrerror2}
h^\frac{1}{2}\|\nabla p^\ast\|_{L^2(\OGamma)}\simeq h\|\nablaG p^\ast\|_{L^2}.
\end{equation}
For  $(\bv_h,\mu_h)\in \bU_h \times Q_h$ we take optimal finite element (nodal or Cl\'ement) interpolants $\bv_h=I_h\big(\bu_T^\ast\big)= I_h((\bu_T^\ast)^e)$, $q_h=I_h(p^\ast)$, and use the notation
$\be_u:=\bu_T^\ast-I_h\big(\bu_T^\ast\big)$, $e_p:=p^\ast-I_h(p^\ast)$. We thus get
\[
  \|\bu_T^\ast -\bu_h,p^\ast -  p_h \| \lesssim \|\be_u\|_U +\|e_p\|_Q + h\|\nablaG p\|_{L^2}+h\|\bu_T^\ast\|_1.
\]
We consider the error term $\|\be_u\|_U$.
Using interpolation properties of piecewise linear polynomials and their traces, cf., e.g., \cite{reusken2015analysis}, 
we have
\begin{equation} \label{estintu}
\begin{split}
  \|\be_u\|_U  &   \lesssim  \|\be_u\|_{V_\ast} + h^\frac12 \|\be_u\|_{H^1(\OGamma)}  \lesssim \|\be_u\|_1 + h^{-1}\|\be_u\|_{L^2} + h^\frac12 \|\be_u\|_{H^1(\OGamma)} \\
&  \lesssim h^{-\frac12} \|\be_u\|_{H^1(\OGamma)} + h^\frac12\|\bu_T^\ast\|_{H^2(\OGamma)}+ h^{-\frac32}\|\be_u\|_{L^2(\OGamma)}   \\ & \lesssim h^\frac12 \|\bu_T^\ast\|_{H^2(\OGamma)} \lesssim h \|\bu_T^\ast\|_{H^2(\Gamma)}.
\end{split}
\end{equation}
Using similar arguments one  derives the bound
\begin{equation} \label{estintp}
 \|e_p\|_Q  \lesssim \|e_p\|_{L^2} + h^\frac12 \|e_p\|_{H^1(\OGamma)} \lesssim  h^\frac12\|p^\ast\|_{H^1(\OGamma)}  \lesssim  h\|p^\ast\|_{H^1(\Gamma)}.
\end{equation}
The combination of these estimates yields the desired result.
\end{proof}
\ \\
\begin{corollary}\label{cor1} Let $(\bu_T^\ast ,p^\ast )$ and $(\bu_h,p_h)$ be as in Theorem~\ref{thm2}. The following discretization error estimates hold:
\begin{align}\label{discrboundA}
 \|\bu_T^\ast  -(\bu_h)_T\|_1 + \|p^\ast  - p_h \|_{L^2}  &\lesssim h (\|\bu_T^\ast \|_{2} +\|p^\ast \|_{1}), \\
  \|\bu_h\cdot\bn \|_{L^2}  &\lesssim h^2 (\|\bu_T^\ast \|_{2} +\|p^\ast \|_{1}).\label{discrbound2}
\end{align}
\end{corollary}
\begin{proof} Note that for $\bv \in U$ we have $\|\bv\|_U^2= A_h(\bv,\bv) \geq a_\tau(\bv,\bv) \geq \|\bv_T\|_1^2$, and for $q \in Q$ we have $\|q\|_Q \geq \|q\|_{L^2}$. Using these estimates and the result in \eqref{discrbound} we obtain
\eqref{discrboundA}. We also have $\|\bv\|_U^2= A_h(\bv,\bv) \geq a_\tau(\bv,\bv) \geq \tau \|v_N\|_{L^2}^2$. Combining this with $\tau=c_\tau h^{-2}$ and the result in \eqref{discrbound} we obtain the bound \eqref{discrbound2}.
\end{proof}

\subsection{$L^2$-error bound} \label{s_L2}
In this section we use a duality argument to derive an optimal $L^2$-norm discretization error bound, based on a  regularity assumption for the problem \eqref{strongform-1}-\eqref{strongform-2}.
We assume that the solution $(\bu=\bu_T,p)$ of the surface Stokes problem \eqref{strongform-1}-\eqref{strongform-2} satisfies the regularity estimate:
\begin{equation}\label{regul}
\|\bu_T\|_2+\|p\|_1 \lesssim \|\blf\|_{L^2},
\end{equation}
for any $\blf\in L^2(\Gamma)^3$, $\blf\cdot\bn=0$, and $g=0$. Again $\|\cdot\|_2$ denotes the $H^2(\Gamma)$ Sobolev norm.

\begin{theorem} \label{thmL2} Let $(\bu_T^\ast ,p^\ast )$ be the solution of \eqref{Stokesweak1_1}-\eqref{Stokesweak1_2} and assume that $(\bu_T^\ast ,p^\ast ) \in (H^2(\Gamma))^3 \times H^1(\Gamma)$. Let  $(\bu_h,p_h) \in \bU_h\times Q_h$ be the solution  of \eqref{discrete}. 
The following error estimate holds:
\begin{equation} \label{L2error}
 \|\bu_T^\ast  - (\bu_h)_T\|_{L^2} \lesssim h^{2} \big(\|\bu_T^\ast \|_{2} + \|p^\ast \|_{1}\big).
\end{equation}
\end{theorem}
\begin{proof}
We consider \eqref{Stokesweak1_1}-\eqref{Stokesweak1_2} with $\blf:=\bu_T^\ast-(\bu_h)_T=:\be_{h,T}$ and $g=0$. Note that $\blf \cdot \bn=0$ on $\Gamma$. The unique solution of this problem is denoted by $(\bw_T^\ast,r^\ast) \in \bV_T \times L_0^2(\Gamma)$. Due to the regularity assumption the $ \bV_T \times L_0^2(\Gamma)$ pair solves also \eqref{strongform-1}-\eqref{strongform-2}, and we have
\begin{equation} \label{reg1}
 \|\bw_T^\ast\|_{2}+\|r^\ast\|_{1}\lesssim \|\be_{h,T}\|_{L^2}.
\end{equation}
The normal extensions of the solution pair are also denoted by $\bw_T^\ast=(\bw_T^\ast)^e$, $r^\ast=(r^\ast)^e$. The consistency property \eqref{consis1} yields
\[
 \bA_h\big((\bw_T^\ast,r^\ast),(\bv,q)\big) = (\be_{h,T}, \bv)_{L^2} + \int_\Gamma E_s(\bw_T^\ast):v_N \bH \, ds - s_h(r^\ast, q) \quad \forall ~(\bv,q) \in \bU \times Q.
\]
Note that the bilinear form $\bA_h(\cdot,\cdot)$ is symmetric.
We take $(\bv,q)=(\bu_T^\ast-\bu_h,p^\ast-p_h) \in \bU \times Q$, which in combination with \eqref{consis2} yields
\begin{align}
 & \|\be_{h,T}\|_{L^2}^2= (\be_{h,T}, \bu_T^\ast-\bu_h)_{L^2}  \nonumber \\
 & =  \bA_h\big((\bw_T^\ast,r^\ast),(\bu_T^\ast- \bu_h,p^\ast-p_h)\big) -\int_\Gamma E_s(\bw_T^\ast):\bH \left(\bn\cdot(\bu_T^\ast- \bu_h)\right)  \, ds \nonumber \\
 &\quad + s_h(r^\ast,p^\ast-p_h ) \nonumber \\
 & =  \bA_h\big((\bu_T^\ast- \bu_h,p^\ast-p_h),(\bw_T^\ast,r^\ast)\big) + \int_\Gamma E_s(\bw_T^\ast):\bH (\bn \cdot \bu_h)\, ds + s_h(r^\ast,p^\ast-p_h ) \nonumber \\
 &= \bA_h\big((\bu_T^\ast- \bu_h,p^\ast-p_h),(\bw_T^\ast- \bw_h,r^\ast- r_h)\big) \label{term1} \\
  & \quad +  \int_\Gamma E_s(\bw_T^\ast): (\bn \cdot \bu_h) \bH \, ds + \int_\Gamma E_s(\bu_T^\ast):(\bn \cdot \bw_h) \bH \, ds \label{term2} \\
  & \quad + s_h(r^\ast,p^\ast-p_h ) - s_h(p^\ast,r_h), \label{term3}
\end{align}
with $\bw_h:=I_h(\bw_T^\ast) \in \bU_h$, $r_h:=I_h(r^\ast)\in Q_h$ optimal piecewise linear interpolations of the extended solution $(\bw_T^\ast)^e$, $(r^\ast)^e$.
We consider the terms in \eqref{term1}-\eqref{term3}. For the term in \eqref{term1} we use continuity of $\bA_h(\cdot,\cdot)$, the discretization error bound \eqref{discrbound}, interpolation error  bounds as in \eqref{estintu}, \eqref{estintp}, and  the regularity estimates \eqref{regul}, \eqref{reg1}:
\begin{equation} \label{hulp1}
\begin{split}
  & | \bA_h\big((\bu_T^\ast- \bu_h,p^\ast-p_h),(\bw_T^\ast- \bw_h,r^\ast- r_h)\big)|  \\ & \leq 2 \big(\|\bu_T^\ast- \bu_h\|_U +\|p^\ast-p_h\|_Q\big)\big( \|\bw_T^\ast- \bw_h\|_U +\|r^\ast- r_h\|_Q \big) \\
 & \lesssim h^2 \big(\|\bu_T^\ast \|_{2} + \|p^\ast \|_{1}\big) \|\be_{h,T}\|_{L^2}.
\end{split}
\end{equation}
For the term in \eqref{term2} we introduce $\be_w:=\bw_T^\ast-\bw_h$.  Using $\bn\cdot \bw_h = - \bn \cdot \be_w$, the discretization error bound \eqref{discrbound2}, interpolation error  bounds as in \eqref{estintu},
 and  the regularity estimates \eqref{regul}, \eqref{reg1}, we get:
\begin{equation} \label{hulp2}
 \begin{split}
  & \left|\int_\Gamma E_s(\bw_T^\ast): (\bn \cdot \bu_h) \bH \, ds + \int_\Gamma E_s(\bu_T^\ast):(\bn \cdot \bw_h) \bH \, ds\right| \\
 & \lesssim \|\bw_T^\ast\|_1\|\bn\cdot \bu_h\|_{L^2} + \|\bu_T^\ast\|_1\|\be_w\|_{L^2}
  \lesssim h^2 \big(\|\bu_T^\ast \|_{2} + \|p^\ast \|_{1}\big) \|\be_{h,T}\|_{L^2}.
\end{split}
\end{equation}
For the term in \eqref{term3} we use the estimates \eqref{normal}, the $H^1$-boundedness of the interpolation operator $I_h$,  the discretization error bound \eqref{discrbound},  and  the regularity estimates \eqref{regul}, \eqref{reg1}:
\begin{equation} \label{hulp3}
 \begin{split}
 & \big| s_h(r^\ast,p^\ast-p_h ) - s_h(p^\ast,r_h)\big|  \\ & \lesssim h \|\nabla r^\ast\|_{L^2(\OGamma)} \|\nabla(p^\ast-p_h)\|_{L^2(\OGamma)}+ h \|\nabla p^\ast\|_{L^2(\OGamma)}\|\nabla I_h(r^\ast)\|_{L^2(\OGamma)} \\
 & \lesssim h \|\nabla_\Gamma r^\ast\|_{L^2} \|p^\ast- p_h\|_Q + h^{\frac32} \|\nabla_\Gamma p^\ast\|_{L^2} \|r^\ast\|_{H^1(\OGamma)} \\
& \lesssim h^2 \big(\|\bu_T^\ast \|_{2} + \|p^\ast \|_{1}\big)\|\be_{h,T}\|_{L^2} + h^2 \|p^\ast\|_1 \|r^\ast\|_1 \\
&\lesssim   h^2 \big(\|\bu_T^\ast \|_{2} + \|p^\ast \|_{1}\big)\|\be_{h,T}\|_{L^2}.
 \end{split}
\end{equation}
Using estimates \eqref{hulp1}, \eqref{hulp2}, \eqref{hulp3} in \eqref{term1}, \eqref{term2}, and \eqref{term3} we obtain
error bound \eqref{L2error}.
\end{proof}

\section{Condition number estimate and algebraic solver} \label{s_cond}
It is well-known \cite{OlshanskiiReusken08,burman2016cutb} that for unfitted finite element methods
there is an issue concerning algebraic stability. In fact, the matrices that represent the discrete problem
may have very bad conditioning due to small cuts in the geometry.
One way to remedy this stability problem is by using
stabilization methods. See, e.g., \cite{burman2016cutb,olshanskii2016trace}.
In this section we show that the `volume normal derivative' stabilizations
in the bilinear forms $A_h(\cdot,\cdot)$ in \eqref{Ah} and $s_h(\cdot,\cdot)$
in \eqref{sh}, with scaling as in \eqref{parameters}, remove any possible algebraic instability.
More precisely,  we show that \emph{the condition number of the stiffness matrix} corresponding to the saddle point problem \eqref{discrete} \emph{is bounded by} $ch^{-2}$, \emph{where the constant $c$ is independent of the position of the interface}. Furthermore, we present an \emph{optimal Schur complement preconditioner}.

Let integer $n>0, m>0$ be the number of active degrees of freedom in $\bU_h$ and $Q_h$ spaces, i.e., $n= {\rm dim}(\bU_h)$, $m={\rm dim}(M_h)$, and $P_h^Q:\,\mathbb{R}^n\to \bU_h$ and $P_h^Q:\,\mathbb{R}^m\to Q_h$ are canonical mappings between the vectors of nodal values and finite element functions.
Denote by $\la\cdot,\cdot\ra$ and $\|\cdot\|$ the Euclidean scalar product and the norm. For matrices, $\|\cdot\|$ denotes the spectral norm in this section.

Let us introduce several matrices. Let
$A, M_u\in\mathbb{R}^{n\times n}$, $B\in\mathbb{R}^{n\times m}$, $C, M_p\in \mathbb{R}^{m \times m}$ be such that
\[
\begin{split}
\la A \vec u, \vec v\ra &=  A_h(P_h^U \vec u, P_h^U \vec v),~ \la B \vec u,\vec \lambda \ra= b_T(P_h^U \vec u,P_h^Q \vec\lambda),~ \la C \vec \lambda,\vec \mu \ra= s_h(P_h^Q \vec \lambda,P_h^Q \vec \mu),
\\
\la M_u \vec u, \vec v\ra&=  (P_h^U \vec u, P_h^U \vec v)_{L^2(\OGamma)},\quad \la M_p\vec \lambda,\vec \mu \ra= (P_h^Q \vec \lambda,P_h^Q \vec \mu)_{L^2(\OGamma)},\quad \quad
\\
\la S_Q \vec \lambda , \vec \mu \ra &= (P_h^Q \vec \lambda,P_h^Q \vec \mu)_{L^2}+ h \big( \nabla (P_h^Q \vec \lambda), \nabla(P_h^Q \vec \mu)\big)_{L^2(\OGamma)}
\end{split}
\]
for all $\vec u,\vec v\in\mathbb{R}^n,~~\vec \mu,\,\vec \lambda\in\mathbb{R}^m$. Note that the mass matrices $M_u$ and $M_p$ do not depend on how the surface $\Gamma$  intersects the domain $\OGamma$. Since the family of background meshes is shape regular, these mass matrices have a spectral condition number that is uniformly bounded,  independent of  $h$ and  of how $\Gamma$ intersects the background triangulation $\mathcal{T}_h$. Furthermore, for  the symmetric positive definite matrix $S_Q$ we have
\[ \la S_Q \vec \lambda, \vec \lambda \ra =\|P_h^Q \vec \lambda\|_Q^2 \quad \text{for all}~\vec \lambda\in\mathbb{R}^m,
\]
cf.~\eqref{defmuast}.
We also introduce the system matrix and its Schur complement:
\[
\A:=\left[\begin{matrix}
             A & B^T  \\
             B & -C
           \end{matrix}\right],\quad S:=B A^{-1} B^T +C.
%
\]
The algebraic system resulting from the finite element method \eqref{discrete} has the form
\begin{equation}\label{SLAE}
\A \vec x=\vec b,\quad\text{with some}~\vec x, \vec b\in\mathbb{R}^{n+m}.
\end{equation}

We will propose a block-diagonal preconditioner of the matrix $\A$. We start by analyzing
preconditioners for matrices $A$ and $S$. In the following lemma we make use of
spectral inequalities for symmetric matrices.
\begin{lemma} \label{lemprecond} There are strictly positive constants $\nu_{A,1}$, $\nu_{A,2}$, $\nu_{S,1}$, $\nu_{S,2}$, $\tilde \nu_{S,1}$, $\tilde \nu_{S,2}$, independent of $h$ and of how $\Gamma$ intersects $\mathcal{T}_h$ such that the following spectral inequalities hold:
 \begin{align}
  \nu_{A,1} h^{-1} M_u & \leq A \leq \nu_{A,2} h^{-3} M_u, \label{spec1} \\
  \nu_{S,1} h^{-1} M_p & \leq S \leq \nu_{S,2} h^{-1} M_p, \label{spec2} \\
  \tilde \nu_{S,1}  S_Q & \leq S \leq \tilde \nu_{S,2} S_Q.  \label{spec3}
 \end{align}
\end{lemma}
\begin{proof}
Note that
\begin{equation}\label{cond1}
\frac{\la A\vec v,\vec v\ra}{\la M_u \vec v, \vec v\ra}=
\frac{A_h(P_h^U \vec v, P_h^U \vec v)}{\|P_h^U \vec v\|^2_{L^2(\OGamma)}}\quad\text{for all}~ \vec v\in\mathbb{R}^n.
  \end{equation}
Let $\bv_h=P_h^U \vec v$. From \eqref{equiv1} we get
\[ \nu_{A,1}h^{-1}\le\frac{A_h(\bv_h, \bv_h)}{\|\bv_h\|^2_{L^2(\OGamma)}} \quad \text{for all}~\bv_h \in \bU_h, 
\]
which proves the lower bound in \eqref{spec1}.
For the upper bound we use \eqref{fund1B} (componentwise) and a finite element inverse estimate,
\begin{align*}
  & \|(\bv_h)_T\|_1^2 + h^{-2} \|(\bv_h)_N\|_{L^2}^2  \leq  \|\bv_h\|_1^2 + h^{-2} \|\bv_h\|_{L^2}^2  \\ & \lesssim h^{-1}\|\bv_h\|_{H^1(\OGamma)}^2+ h^{-3}\|\bv_h\|_{L^2(\OGamma)}^2 \lesssim h^{-3}\|\bv_h\|_{L^2(\OGamma)}^2 \quad \text{for all}~\bv_{h} \in \bU_h.
\end{align*}
Combining this with \eqref{equiv1}, we obtain
\[
 \frac{A_h(\bv_h, \bv_h)}{\|\bv_h\|^2_{L^2(\OGamma)}} \leq   \nu_{A,2}h^{-3}~~\text{for all}~\bv_h \in \bU_h,
\]
for a suitable constant $\nu_{A,2}$. This proves the  second inequality in \eqref{spec1}.
For the Schur complement matrix ${S}=B A^{-1} B^T +C $, we first note that
\begin{equation}\label{cond3}
 \la BA^{-1} B^T \vec \lambda,\vec \lambda \ra = \Big( \sup_{\bv_h \in \bU_h} \frac{b_T(\bv_h,\mu_h)}{\|\bv_h\|_U} \Big)^2, \quad \mu_h:=P_h^Q \vec \lambda.
  \end{equation}
Using the discrete inf-sup estimate \eqref{infsupest} and \eqref{fundC}, we thus get
\begin{equation} \label{inf1}
\begin{split}
 \la S\vec \lambda,\vec \lambda\ra &  = \Big( \sup_{\bv_h \in \bU_h} \frac{b_T(\bv_h,\mu_h)}{\|\bv_h\|_U} \Big)^2 + s_h(\mu_h,\mu_h)  \\ & \gtrsim  \|\mu_h\|_Q^2 \gtrsim h^{-1}\|\mu_h\|_{L^2(\OGamma)}^2 \gtrsim h^{-1} \la M_p \vec \lambda,\vec \lambda\ra ,
\end{split} \end{equation}
which proves the first inequalities in \eqref{spec2} and \eqref{spec3}.
Using
\[
   \frac{b_T(\bv_h,\mu_h)}{\|\bv_h\|_U} \leq \frac{\|(\bv_h)_T\|_1 \|\mu_h\|_{L^2}}{\|\bv_h\|_U} \leq \|\mu_h\|_{L^2}
\]
and \eqref{fundC} we also obtain
\[ \begin{split}
  \la S\vec \lambda,\vec \lambda\ra  & \leq \|\mu_h\|_{L^2}^2 + s_h(\mu_h,\mu_h) \leq \max\{ 1, c_p\} \|\mu_h\|_Q^2   \\ & \lesssim h^{-1} \|\mu_h\|_{L^2(\OGamma)}^2 \lesssim h^{-1} \la M_p \vec \lambda,\vec \lambda\ra,
\end{split} \]
 which proves the second inequalities in \eqref{spec2} and \eqref{spec3}.
\end{proof}
\ \\[1ex]
The results in \eqref{spec2} and \eqref{spec3} yield that both the pressure mass matrix $M_p$ and the matrix $S_Q$ are optimal preconditioners for the Schur complement matrix $S$. This is an analog of a well-known result for (stabilized) finite element discretizations of the Stokes problem in Euclidean spaces.

We introduce a block diagonal preconditioner
\begin{equation}\label{block}
  Q:=\left[\begin{matrix}
             Q_A & 0  \\
             0 & Q_S
           \end{matrix}\right]
\end{equation}
for $\A$.
In order to analyze it, we can apply analyses known from the literature, e.g. section 4.2 in \cite{ElmanBook}.
\begin{corollary}\label{corr:pc}
The following estimate holds for some $c>0$ independent of $h$ and of how $\Gamma$ cuts through the background mesh:
\begin{equation}\label{Condest}
\mathrm{cond}(\A)=\|\A\| \|\A^{-1}\|\le c\,h^{-2}.
\end{equation}
\end{corollary}
\begin{proof}
Take $Q_A:=M_u$, $Q_S:=M_p$. We can apply Theorem~4.7 from \cite{ElmanBook} with (notation from \cite{ElmanBook}) preconditioners $\bP=M_u$, $T=M_p$. This yields that all eigenvalues of $Q^{-1}\A$ are contained in the union of intervals
\begin{equation}\label{EigBound}
[-c_0 h^{-1}, -c_1 h^{-1}] \cup [d_0 h^{-1}, d_1,h^{-3}] ,
\end{equation}
with constants $c_0 > c_1 > 0$, $0 < d_0 < d_1$ that depend only on the constants $ \nu_{A,i}$, $\nu_{S,i}$ in \eqref{spec1}, \eqref{spec2}.
From this spectral estimate and the  fact that $Q$ has a uniformly bounded condition number we conclude that \eqref{Condest} holds.
\end{proof}
\ \\[2ex]
The application of Theorem~4.7 from \cite{ElmanBook} also yields the following result.
\begin{corollary} \label{cor2}
Let $Q_A \sim A$ be a uniformly spectrally equivalent preconditioner of $A$ and $Q_S:=M_p$ or $Q_S:=S_Q$. For the spectrum $\sigma(Q^{-1}\A)$
of the preconditioned matrix  we have
\[ \sigma(Q^{-1}\A) \subset \big([C_{-},c_{-}]\cup[c_+,C_+]\big), \]
with some constants
$C_{-} < c_{-} < 0 < c_+ < C_+$ independent of $h$ and the position of $\Gamma$.
\end{corollary}

In section \ref{steady_Stokes} we study the performance of preconditioner \eqref{block} with $Q_A = A$ and $Q_S = S_Q$.

\section{Numerical experiments} \label{sectExp}

A series of numerical tests is presented to showcase the main features and performance
of the TraceFEM  for the surface Stokes problem.

As explained in Remark \ref{remimplementation}, we approximate the surface
$\Gamma$ with  a piecewise planar  approximation $\Gamma_h$,
with ${\rm dist}(\Gamma,\Gamma_h) \lesssim h^2$.
We use piecewise linear finite elements for both
velocity and pressure in problem \eqref{discrete}.
Higher order finite elements are possible (see, e.g., \cite{grande2017higher})
and will be addressed in a forthcoming paper.

We remind that a second source of geometric error is given by the
approximation $\bn_h$ of the exact normal $\bn$. For the numerical results
below, we choose an approximation
$\bn_h=\frac{\nabla\phi_h}{\|\nabla\phi_h\|_2}$, where
where $\phi_h$ is defined as a $P_2$ nodal interpolant of the level set function.

We first consider the Stokes problem, i.e.~\eqref{strongform-1}-\eqref{strongform-2}
 on the unit sphere. To satisfy assumption~\ref{A1}, we set $\alpha = 1$. The goals of this first test are:
to check the spatial accuracy of the TraceFEM, thereby verifying numerically
the theoretical results in Corollary~\ref{cor1} and Theorem~\ref{thmL2};
to check the sensitivity of the spatial discretization error with respect to the pressure
stabilization parameter $\rho_p=c_ph$; and to illustrate the behavior of a preconditioned MINRES solver.
Regarding other parameters (cf. \eqref{parameters}) we note the following:   different choices of the parameter $\rho_u$ were analysed in \cite{burman2016cutb,grande2017higher,gross2017trace}, and the TraceFEM was found to be
very robust with respect to the variation of this parameter. Based on these experiences  we simply set  $\rho_u:=h$. Regarding the penalty parameter $\tau$, previous numerical studies of vector surface problems  \cite{hansbo2016stabilized,hansbo2016analysis,reuther2017solving} all suggest that $\tau$ should be
``sufficiently'' large. This is consistent with our experience, which shows that the approach is largely insensitive to the variation of $\tau$ once it is large enough.

Next, we consider in section~\ref{unsteady_Stokes} the unsteady Stokes problem discretized in
time by the backward Euler method. With this second test we want to illustrate
the expected evolution of the flow to the reference Killing field for different meshes
and different time steps. Finally, in section \ref{test3} we show a flow field computed on an implicitly given
manifold with (strongly) varying curvature.

\subsection{Stokes problem on the unit sphere}\label{steady_Stokes}


The surface $\Gamma$  is the unit sphere,
centered at the origin.
We characterize it as the zero level of the level set function $\phi(\bx) = \|\bx\|_2 -1$,
where $\bx = (x_1, x_2, x_3)^T$.
We consider the following exact solution to problem \eqref{strongform-1}-\eqref{strongform-2}:
\begin{equation} \label{exact}
\begin{split}
\bu^*&=\bP(-x_3^2,x_2,x_1)^T \in\bV_T, \\
p^* &= x_1x_2^3+x_3 
\in\ L_0^2(\Gamma).  \\
\end{split}
\end{equation}
The forcing term $\mathbf{f}$ in eq.~\eqref{strongform-1} and
source term $g$ in eq.~\eqref{strongform-2} are readily computed from the
above exact solution.
Notice that the pressure average is zero.
We set $\alpha = 1$ to exclude the Killing vectors on a sphere
from the kernel.
The sphere is embedded in an outer cubic domain $\Omega=[-5/3,5/3]^3$.
The triangulation $\T_{h_\ell}$ of $\Omega$ consists of $n_\ell^3$ sub-cubes,
where each of the sub-cubes is further refined into 6 tetrahedra.
Here $\ell\in\Bbb{N}$ denotes the level of refinement, with the associated
mesh size $h_\ell= \frac{10/3}{n_\ell}$ and $n_\ell= 2^{\ell+1}$.
Parameters $\tau$, {$\rho_p$, and $\rho_u$ are set as in \eqref{parameters}
with values of $c_\tau$, $c_p$, and $c_u $}
that will be specified for each case.
The velocity and pressure computed with the mesh associated to
refinement level $\ell=5$ are illustrated in Figure~\ref{fig:test8_solution}.

\begin{figure}[htp!]
  \centering
  \subfloat[Velocity]{\includegraphics[width=.48\textwidth]{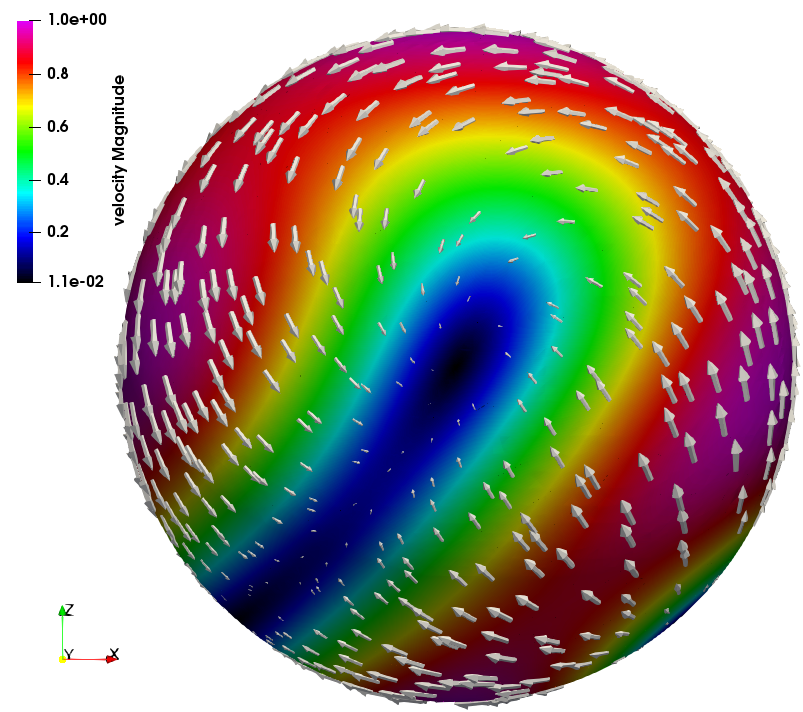}}
  \subfloat[Pressure]{\includegraphics[width=.48\textwidth]{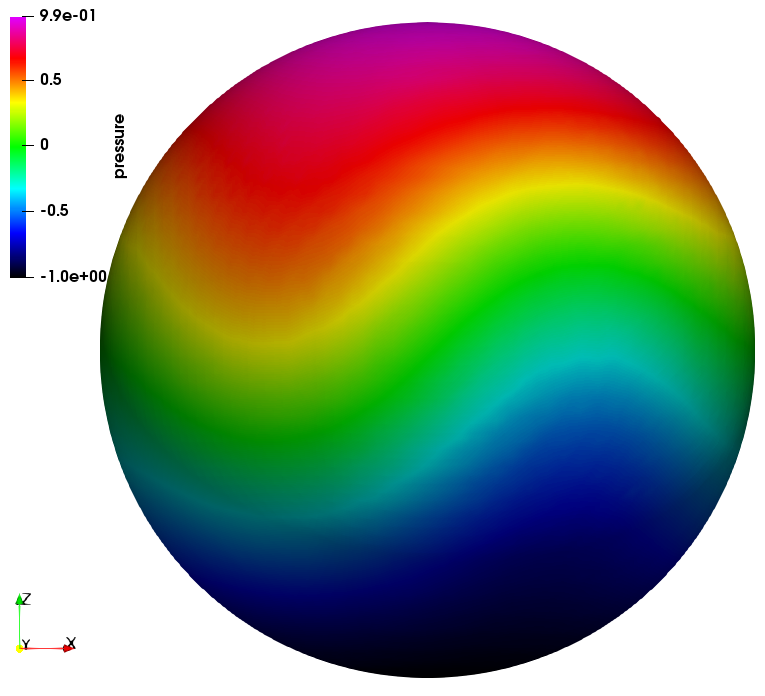}}
  \hfill
  \caption{Velocity and pressure computed on refinement level $\ell = 5$ mesh.}
  \label{fig:test8_solution}
\end{figure}


We report in Fig.~\ref{fig:test8_h-convergence} the $L^2(\Gamma)$ and $H^1(\Gamma)$ norms  of the error for the velocity, $L^2(\Gamma)$ norm of the normal velocity component,
and the $L^2(\Gamma)$ norm of the error for the pressure plotted against the refinement level $\ell$. All
the norms reported in Fig.~\ref{fig:test8_h-convergence}, {and also in Fig.~\ref{fig:test8_alpha_convergence} and
\ref{fig:test10_kin}}, \emph{are computed on the approximate surface} $\Gamma_h$, {while for the exact velocity and pressure solutions (on $\Gamma$) we use the continuous extension as specified in \eqref{exact}.}
We observe that optimal convergence orders are achieved for  $\|\bu^*-\bu_h\|_{H^1(\Gamma)}$
and $\|\bu^* -\bP\bu_h\|_{L^2(\Gamma)}$, as predicted by the theoretical results in
Corollary~\ref{cor1} and Theorem~\ref{thmL2}. We note though that the theoretical analysis
does not account for the geometric errors.
For $\|p  - p_h \|_{L^2(\Gamma)}$, we see a faster convergence than predicted by
Corollary~\ref{cor1}.
The results shown in Fig.~\ref{fig:test8_h-convergence} have been obtained
for $c_\tau = 1$, { $c_p =1$, and $c_u = 1$}.

\begin{figure}
\begin{center}
  \begin{tikzpicture}
   \def\vara{100}
   \def\varb{1}
   \begin{semilogyaxis}[ xlabel={Refinement level $\ell$}, ylabel={Error}, ymin=1E-4, ymax=100, legend style={ cells={anchor=west}, legend pos=south west} ]
    \addplot+[red,mark=o,mark size=3pt,line width=1.5pt] table[x=level, y=L2v] {test8_h.dat};
    \addplot+[blue,mark=+,mark size=3pt,line width=1.5pt] table[x=level, y=H1v] {test8_h.dat};
    \addplot+[black!30!yellow,mark=diamond,mark size=3pt,line width=1.5pt] table[x=level, y=L2vT] {test8_h.dat};
    \addplot+[black!30!green,mark=x,mark size=3pt,line width=1.5pt] table[x=level, y=L2p] {test8_h.dat};
    \addplot[dashed,line width=1pt] coordinates { 
      (0,\vara) (1,\vara*0.5) (2,\vara*0.25) (3,\vara*0.125) (4,\vara*0.0625) (5,\vara*0.03125) (6,\vara*0.03125/2)
    };
    \addplot[dotted,line width=1pt] coordinates { 
     (0,\varb) (1,\varb*0.5*0.5) (2,\varb*0.25*0.25) (3,\varb*0.125*0.125) (4,\varb*0.0625*0.0625) (5,\varb*0.03125*0.03125) (6,\varb*0.03125*0.03125*0.5*0.5)
    };
    \legend{$\|\bu^* - \bu_h\|_{L^2(\Gamma)}$, $\|\bu^* -\bu_h\|_{H^1(\Gamma)}$, $\|\bu_h \cdot \bn_h \|_{L^2(\Gamma)}$, $\|\bp^* - \bp_h\|_{L^2(\Gamma)}$, $\mathcal{O}(h^1)$, $\mathcal{O}(h^{2})$}
   \end{semilogyaxis}
  \end{tikzpicture}
  \caption{
  $L^2(\Gamma)$ and $H^1(\Gamma)$ norms of the error for the velocity, $L^2(\Gamma)$ norm of the normal velocity
and the $L^2(\Gamma)$ norm of the error for the pressure plotted against the refinement level $\ell$.
The computational results have been obtained with {$\rho_p=h$}.}
  \label{fig:test8_h-convergence}
  \end{center}
\end{figure}
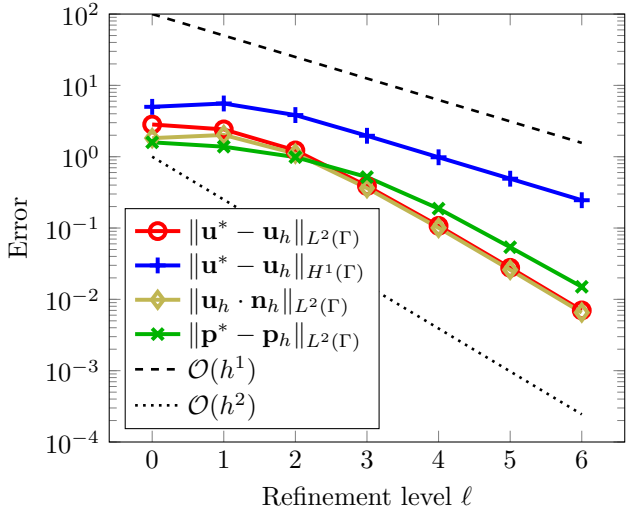

We now vary the value of parameter {$c_p$}. Figure~\ref{fig:test8_alpha_convergence}
shows the $L^2(\Gamma)$ and $H^1(\Gamma)$ norms of the error for the velocity
and the $L^2(\Gamma)$ norm of the error for the pressure for {$c_p \in [5 \cdot 10^{-3}, 10]$}.
We see that as the value of {$c_p$} moves away from 1 the errors
either remain almost constant or slightly increase. Thus among the values of
{$c_p$} that we considered, {$c_p = 1$} is close to an ``optimal'' choice, {and there is a  low sensitivity of the accuracy depending on $c_p$.}

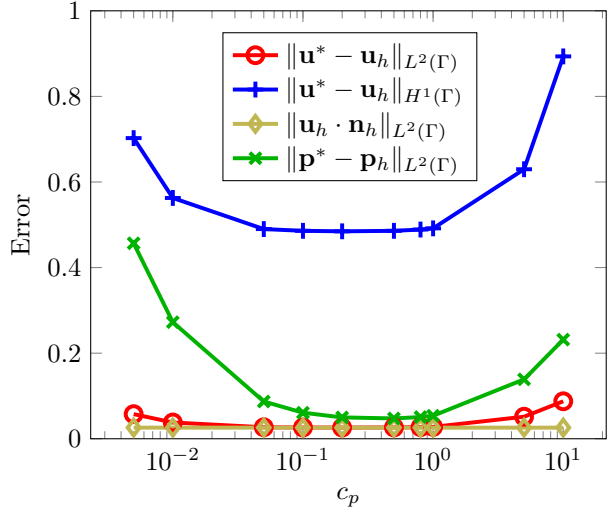
\begin{figure}
\begin{center}
 \begin{tikzpicture}
  \begin{semilogxaxis}[ xlabel={$c_p$}, ylabel={Error}, ymin=0, ymax=1, legend style={ cells={anchor=west}, at={(0.75,0.95)}} ]
   \addplot+[red,mark=o,mark size=3pt,line width=1.5pt] table[x=alpha, y=L2v] {test8_alpha.dat};
   \addplot+[blue,mark=+,mark size=3pt,line width=1.5pt] table[x=alpha, y=H1v] {test8_alpha.dat};
    \addplot+[black!30!yellow,mark=diamond,mark size=3pt,line width=1.5pt] table[x=alpha, y=L2vN] {test8_alpha.dat};
   \addplot+[black!30!green,mark=x,mark size=3pt,line width=1.5pt] table[x=alpha, y=L2p] {test8_alpha.dat};
    \legend{$\|\bu^* - \bu_h\|_{L^2(\Gamma)}$,$\|\bu^* -\bu_h\|_{H^1(\Gamma)}$,$\|\bu_h \cdot \bn_h \|_{L^2(\Gamma)}$,$\|\bp^* - \bp_h\|_{L^2(\Gamma)}$}
  \end{semilogxaxis}
 \end{tikzpicture}
  \caption{
  $L^2(\Gamma)$ and $H^1(\Gamma)$ norms of the error for the velocity, $L^2(\Gamma)$ norm of the normal velocity
and the $L^2(\Gamma)$ norm of the error for the pressure plotted
against the value of $c_p$.
The computational results have been obtained with the $\ell=5$ mesh.}
   \label{fig:test8_alpha_convergence}
   \end{center}
\end{figure}

For the solution of the linear system \eqref{SLAE}, we used the preconditioned
MINRES method  with the block-diagonal preconditioner \eqref{block}.  The preconditioner $Q_A$ is defined
through the application a standard SSOR-preconditioned CG method to solve {$A \vec v = \vec b$} iteratively, with a tolerance such that the initial residual is reduced by a factor of $10^4$.
The same strategy is used to define $Q_S$, i.e.  it is also defined
through the application a standard SSOR-preconditioned CG method to solve {$S_Q \vec p = \vec c$} iteratively, with a tolerance such that the initial residual is reduced by a factor of $10^4$.
As initial guess for the MINRES method we chose the zero vector.
We adopted a stopping criterium based on the Euclidean norm
of the residual of system \eqref{SLAE} and set the stopping tolerance to $10^{-8}$.
We report in Tables~\ref{tab:test8_h_iter_res} and \ref{tab:test8_alpha_iter_res}
the number of MINRES iterations for the simulations in Figures
\ref{fig:test8_h-convergence} and \ref{fig:test8_alpha_convergence}, respectively.

\begin{table}[ht!]
  \centering
    \begin{tabular}{r|ccccccc}
      \toprule
      $\ell$          & 0 &1 &2 &3 &4 &5 &6 \\ \midrule
      $\#$ iterations & 10&14&20&26&29&29&29\\  
average $\#$ inner CG iterations, precond. $A$  &5&8&16&27&51&98&184\\
average $\#$ inner CG iterations, precond. $S_Q$ &6&7&7&8&8&8&8\\

    \end{tabular}
    \caption{Total number of MINRES iterations  for different refinement levels $\ell$ and the average number of inner CG iterations needed to compute the action of the preconditioners.
For all the simulations we set $c_p = 1$.}
    \label{tab:test8_h_iter_res}
\end{table}

\begin{table}[ht!]
  \centering
  \begin{tabular}{r|ccccccc}
      \toprule
     $c_p$          & 0.01 &0.05 &0.1 &0.5 &1 &5 &10 \\ \midrule
      $\#$ iterations & 120&54&39&20&29&64&86\\  
    \end{tabular}
    \caption{Total number of MINRES iterations  for different values of $c_p$.
     All the simulations used the refinement level $\ell = 5$ mesh.}
    \label{tab:test8_alpha_iter_res}
\end{table}

From Table~\ref{tab:test8_h_iter_res} we see that the number of MINRES iterations
grows from refinement level $\ell = 0$ to $\ell =2$,
then it increases slightly till $\ell =4$, and for $\ell \geq 4$ it levels off. This observation is consistent with the result of Corollary~\ref{cor2}. We also report the average number of inner
preconditioned CG iterations needed to compute the matrix--vector product with preconditioners $A$ and $S_Q$.
We see that $S_Q$ is uniformly well-conditioned, while the linear growth of the number of
the preconditioned CG iterations for $A$ is similar to what one expects for a standard
FE discretization of the Poisson problem.

From Table~\ref{tab:test8_alpha_iter_res} we observe that when
using the mesh associated with $\ell = 5$ the number of MINRES iterations
attains its minimum approximately for the same value of  {$c_p$} that minimizes the discretization error.
For $c_p \leq 10^{-3}$, MINRES did not converge within 300 iterations.
For example, for  $c_p=10^{-3}$ (resp.,  $c_p=10^{-4}$) the stopping criterion
is satisfied after 379 (resp., 1182) iterations.
{This may be related to the fact that
the Schur complement is symmetric and positive definite
only if $c_p$ is sufficiently large, cf. Lemma~\ref{lemdiscreteinfsup}.} Hence, the  used preconditioner
is more sensitive to variations of $c_p$ than
to mesh refinement.

From the results in this section we deduce that the straightforward choice {$c_\tau=c_p=c_u = 1$}
is a good compromise between minimizing the spatial discretization error
and keeping the number of MINRES iterations low. Thus, the results in
sections~\ref{unsteady_Stokes} and \ref{complex_geometry} below have been obtained
with {$c_\tau=c_p=c_u = 1$.}

\subsection{Unsteady Stokes problem on the unit sphere}\label{unsteady_Stokes}
In this section, we consider the unsteady incompressible Stokes problem
posed on the surface of a unit sphere and
discretized in time with the implicit Euler method.
This gives rise to problem as in \eqref{strongform-1}-\eqref{strongform-2}
with $\alpha = 1/\Delta t$,
where $\Delta t$ is a time step. 
We assume that no external force is applied
and we set $g = 0$ in equation~\eqref{strongform-2}.
The initial velocity at $t=0$ is chosen to be
\[
\bu_0 = \bn \times{} \nabla{}_{\Gamma}\left( Y^{x_3}_1 +  Y^{x_2}_1 +  Y^{x_3}_2 +Y^{x_3}_3 \right) \,,
\]
where $Y^{x_l}_k$ is a real spherical harmonic of zero order and degree $k$, axisymmetric with respect to the $x_l$ coordinate axis. Notice that $\Div_\Gamma \bu_0=0$.
It can easily be checked that all spherical harmonics have zero angular momentum except for those of first degree.
Since there is only one Killing vector field with given total angular momentum,  we expect the solution to evolve towards the reference Killing vector field with the same non-zero angular momentum as $\bu_0$.
Therefore, the reference Killing vector field corresponds to the rotating motion over the axis $\be_{3} + \be_2$
with angular momentum $\int_{\Gamma{}}\br \times{}  \bn \times{} \nabla{}_{\Gamma}\left( Y^{x_3}_1 +  Y^{x_2}_1  \right)dS$,
where $\br$ is the vector connecting the origin of the axes to point
$\bx$ lying on the unit sphere.

We consider the meshes associated to refinement levels $\ell = 2, 3, 4, 5$
used in Sec.~\ref{steady_Stokes} and two values for the time step $\Delta t = 0.1, 0.01$.
Fig.~\ref{fig:test10} shows both the initial velocity and the
reference Killing vector field computed with mesh $\ell = 5$
and $\Delta t = 0.1$.

\begin{figure}[h!]
  \centering
  \subfloat[Initial velocity]{\includegraphics[width=.48\textwidth]{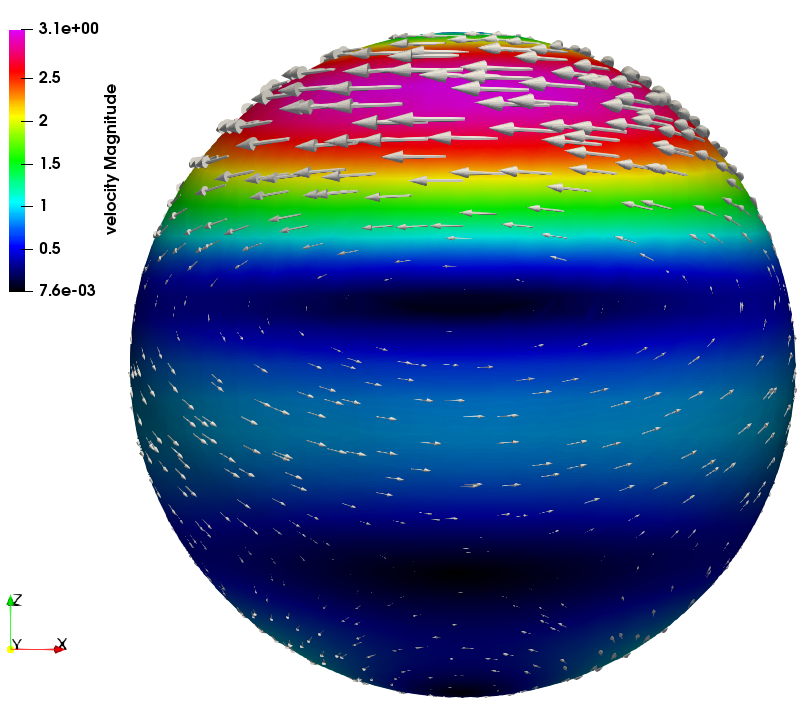}}
  \subfloat[Killing vector field]{\includegraphics[width=.48\textwidth]{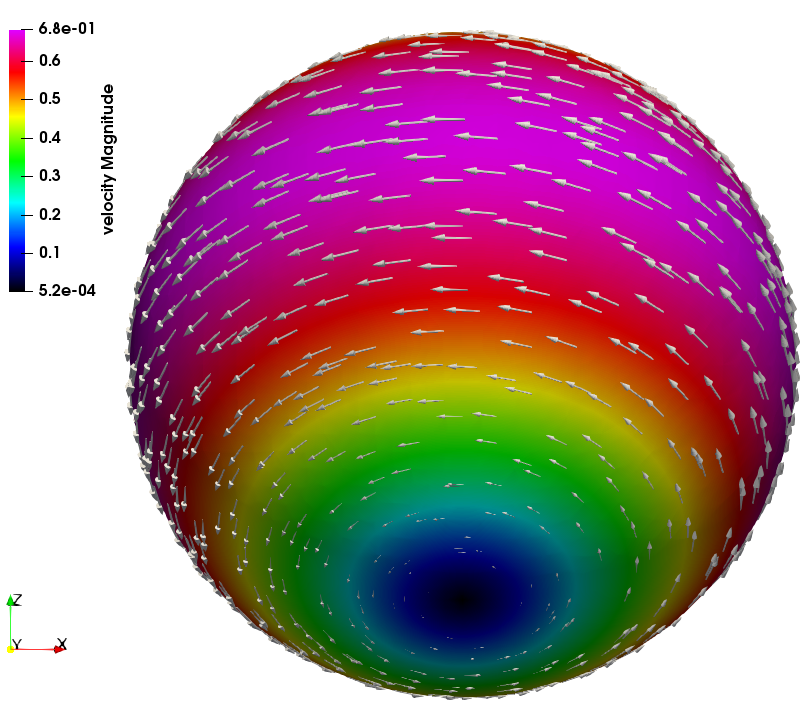}}
  \hfill
  \caption{Initial velocity and reference Killing vector field computed with the mesh  $\ell = 5$  and $\Delta t = 0.1$.}
  \label{fig:test10}
\end{figure}

%

Figure~\ref{fig:test10_kin} shows the evolution of the computed kinetic energy
over time for refinement levels $\ell = 2, 3, 4, 5$. Kinetic energy was calculated after each time step as $0.5\|\bu_h\|^2_{L^2(\Gamma)}$. The reference line corresponds to kinetic energy of the reference Killing vector field.

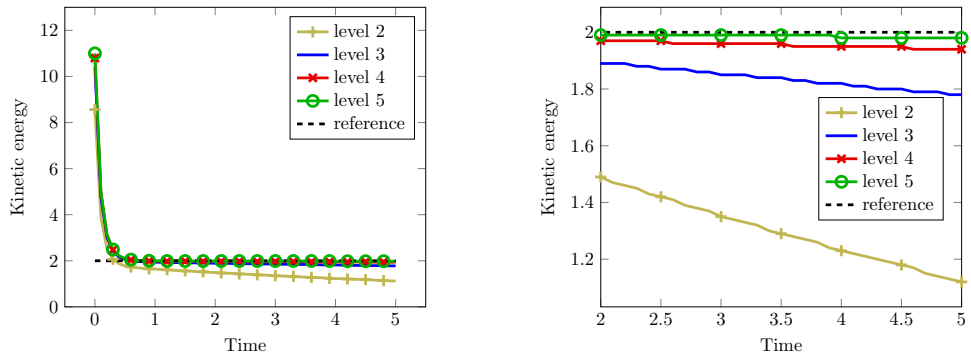
\begin{figure}

 \begin{minipage}{0.45\textwidth}
\begin{center}
  \begin{tikzpicture}[scale=0.7]
   \def\vara{100}
   \def\varb{1}
   \begin{axis}[xlabel={Time}, ylabel={Kinetic energy}, ymin=1E-4, ymax=13, legend style={ cells={anchor=west}, legend pos=north east} ]
    \addplot[black!30!yellow,mark=+,mark size=3pt,mark repeat={3},line width=1.5pt] table[x=time, y=ref2] {kinetic.txt};
    \addplot[blue,mark=,mark size=3pt,mark repeat={3},line width=1.5pt] table[x=time, y=ref3] {kinetic.txt};
    \addplot[black!10!red,mark=x,mark repeat={3},mark size=3pt,line width=1.5pt] table[x=time, y=ref4] {kinetic.txt};
    \addplot[black!30!green,mark=o,mark repeat={3},mark size=3pt,line width=1.5pt] table[x=time, y=ref5] {kinetic.txt};
    \addplot[black,dashed,line width=1.5pt] coordinates {  (0,2)  (5,2)  };
    \legend{level 2, level 3  , level 4 ,   level 5  ,reference}
   \end{axis}
  \end{tikzpicture}
  \end{center}
\end{minipage}
\hfill
\begin{minipage}{0.45\textwidth}
\begin{center}
  \begin{tikzpicture}[scale=0.7]
   \def\vara{100}
   \def\varb{1}
   \begin{axis}[xlabel={Time}, ylabel={Kinetic energy}, xmin=2, xmax=5, legend style={ cells={anchor=west}, , at={(0.95,0.7)}} ]
    \addplot[black!30!yellow,mark=+,mark size=3pt,mark repeat={5},line width=1.5pt] table[x=time, y=ref2] {kinetic.txt};
    \addplot[blue,mark=,mark size=3pt,mark repeat={5},line width=1.5pt] table[x=time, y=ref3] {kinetic.txt};
    \addplot[black!10!red,mark=x,mark repeat={5},mark size=3pt,line width=1.5pt] table[x=time, y=ref4] {kinetic.txt};
    \addplot[black!30!green,mark=o,mark repeat={5},mark size=3pt,line width=1.5pt] table[x=time, y=ref5] {kinetic.txt};
    \addplot[black,dashed,line width=1.5pt] coordinates {  (0,2)  (5,2)  };
    \legend{level 2, level 3  , level 4 ,   level 5  ,reference}
   \end{axis}
  \end{tikzpicture}
   \end{center}
\end{minipage}

 \caption{Time evolution of the kinetic energy computed for different refinement levels and $\Delta t = 0.1$ (left). Zoom-in for time $t = [2,5]$ s (right).}
  \label{fig:test10_kin}
\end{figure}

The time-dependent discrete surface Stokes problem is a dissipative dynamical system. Hence its kinetic energy asymptotically decays exponentially at the rate equal to its minimal eigenvalue. Since the sphere has non-trivial
Killing fields,  for the original differential problem the \emph{asymptotic decay rate} is zero. Although this paper does not analyze eigenvalue convergence, one may expect that the FE problem approximates the asymptotic evolution of the system.
{As an indication for} this, we fit the kinetic energy computed with the meshes associated to refinement levels
$\ell = 2, 3, 4, 5$ and $\Delta t = 0.1, 0.01$ with the function $A\exp{}(-\lambda{}t)$
for time $t \in [2, 5]$. We report in Table~\ref{tab:test10_exp_fitting}
the values of $\lambda$ for each case.
We notice that as the time step goes from 0.1 to 0.01
there is only a small difference in the value of $\lambda$.
From Table~\ref{tab:test10_exp_fitting} we see that as the mesh
gets finer the discrete approximations of the asymptotic decay rate $\lambda$ converge to zero approximately as $O(h^2)$. This convergence is consistent with the second order accuracy of our finite element method.

\begin{figure}
  \centering
    \begin{tabular}{r|cccccccc}
      \toprule
    $\ell$  &   2	&   2	&   3	&   3	&   4	&      4	&  5 & 5\\
    $\Delta t$ & 1e-1  &  1e-2&  1e-1      &  1e-2&  1e-1      &   1e-2&  1e-1  &  1e-2 \\   \midrule
    $\lambda$  & 	9.40e-2	& 	9.75e-2&  2.13e-2 &  2.25e-2&  5.26e-3 &  5.51e-3&  1.64e-3 &  1.65e-3\\
    \end{tabular}
    \caption{Values of $\lambda$ for the exponential fitting with $A\exp{(-\lambda{}t)}$
    of the kinetic energy computed for time $t \in [2, 5]$ with the meshes at refinement levels
$\ell = 2, 3, 4, 5$ and $\Delta t = 0.1, 0.01$.}
    \label{tab:test10_exp_fitting}
\end{figure}

\subsection{Source and sink flow on an implicitly defined surface}\label{complex_geometry}\label{test3}

In this section,
we consider the unsteady surface Stokes problem posed on a more complex manifold.
The surface $\Gamma$ is
implicitly defined as the zero level-set of
\begin{align}
\phi(x_1,x_2,x_3) =&(x_1^2+x_2^2-4)^2+(x_2^2-1)^2+(x_2^2+x_3^2-4)^2+(x_1^2-1)^2 \cl
&+(x_1^2+x_3^2-4)^2+(x_3^2-1)^2-13. \el
\end{align}
The example of $\Gamma$ is taken from~\cite{chernyshenko2015adaptive}.
We embed $\Gamma$ in the cube $[-3,3]^3$ centered at the origin.
We assume that no external force is applied and we start the simulation
from fluid at rest, i.e.~initial velocity is $\bu_0 = \mathbf{0}$.
The flow is driven by the non-zero source term in the mass balance equation~\eqref{strongform-2}:
\[
g(\br)=\frac{1}{h^2}\left(\exp\left({-\frac{\|\br-\ba\|^2}{h^2}}\right)
- \exp\left({-\frac{\|\br-\bb\|^2}{h^2}}\right)\,\right) \,,\]
where $\ba=(-1.0, 1.0,\sqrt{(7 + \sqrt{19})/3})$ and $\bb=(1.0, -1.0 ,-\sqrt{(7 +
\sqrt{19})/3})$ are located on the manifold.
Note that $g$ consists of fluid source and sink, which approximate point source and sink for $h\to0$.

\begin{figure}[htp!]
  \centering
  \subfloat[Velocity]{\includegraphics[width=.48\textwidth]{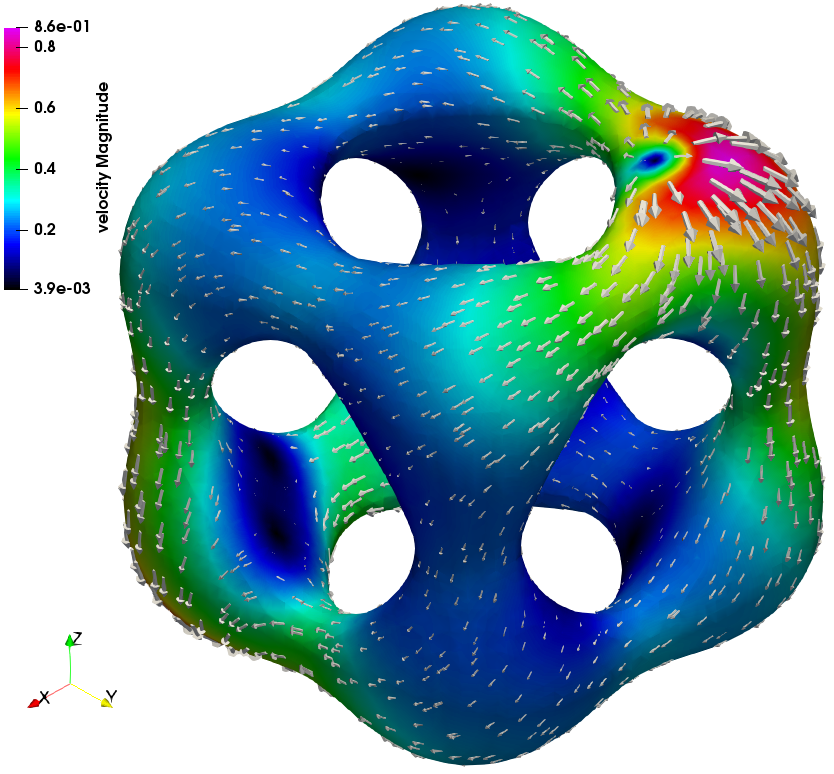}}
  \subfloat[Pressure]{\includegraphics[width=.48\textwidth]{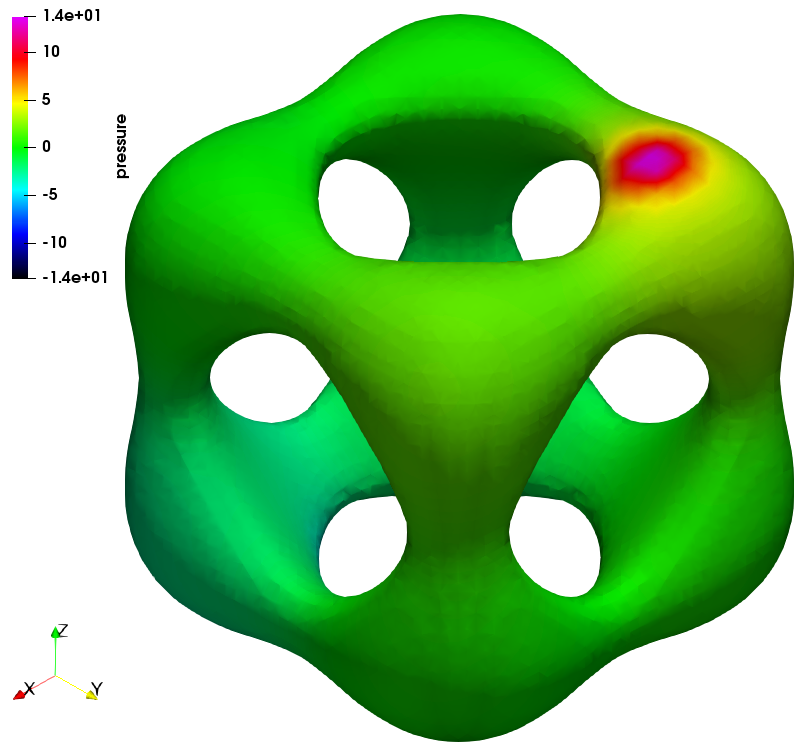}}
  \hfill
  \caption{Velocity and and pressure at $t=10$ computed with  mesh $\ell = 5$ and $\Delta t = 0.1$.}
  \label{fig:test10A}
\end{figure}

We set mesh level $\ell=5$ and time step $\Delta t = 0.1$.  Figure~\ref{fig:test10A}
shows the computed velocity and pressure after they {have (essentially) converged} to an equilibrium state.
The results were obtained for {$c_\tau = 10$, $c_p =1$, and $c_u = 1$}. This example illustrates
the flexibility of the proposed numerical method in handling fluid problems over complex geometries without surface parametrization and mesh fitting.

\bibliographystyle{siam}
\bibliography{literatur}{}

\end{document}